\documentclass{amsart}

\usepackage{curves}
\usepackage{epic}
\usepackage{eepic}
\usepackage{amsmath}
\usepackage{amscd}
\usepackage{amsfonts}
\usepackage{amssymb}
\usepackage{latexsym}
\usepackage{pifont}
\usepackage{array}
\usepackage{oldgerm}

\newtheorem{thm}{Theorem}[section]
\newtheorem{pro}[thm]{Proposition}
\newtheorem{lem}[thm]{Lemma}
\newtheorem{cor}[thm]{Corollary}
\newtheorem{thm&def}[thm]{Theorem \& Definition}
\newtheorem{lem&def}[thm]{Lemma \& Definition}

\theoremstyle{definition}
\newtheorem{defi}[thm]{Definition}

\newtheorem{exa}[thm]{Example}

\theoremstyle{remark}
\newtheorem{rmk}[thm]{Remark}

\newcommand{\ncm}{\newcommand}

\ncm{\Cat}{\mathsf{Cat}}
\ncm{\CAT}{\mathsf{CAT}}
\ncm{\ACT}{\mathsf{ACT}}
\ncm{\ob}{\operatorname{ob}}
\ncm{\Nat}{\operatorname{Nat}}
\ncm{\Set}{\mathsf{Set}}
\ncm{\Ab}{\mathsf{Ab}}
\ncm{\Add}{\mathsf{Add}}
\ncm{\Mnd}{\mathsf{Mnd}}
\ncm{\Cmd}{\mathsf{Cmd}}
\ncm{\Fun}{\mathsf{Fun}}
\ncm{\Mon}{\mathsf{Mon}}
\ncm{\Elt}{\mathsf{Elt}\,}
\ncm{\Sub}{\mathsf{Sub}\,}
\ncm{\Flat}{\mathsf{Flat}}
\ncm{\add}{\text{\Large\textsl{a}}}
\ncm{\proj}{\mathsf{proj}}
\ncm{\Col}{\mathsf{Col}\,}
\ncm{\fgmod}[1]{\mathsf{mod}\text{-}#1}
\ncm{\UEE}{\mathsf{u\text{-}eff}}
\ncm{\EEE}{\mathsf{e\text{-}eff}}
\ncm{\HEE}{\mathsf{h\text{-}eff}}
\ncm{\EE}{\mathsf{eff}}
\ncm{\Split}{\mathsf{split}}
\ncm{\Cl}{\mathsf{Cl}}
\ncm{\ord}{\mathsf{ord}}
\ncm{\Mod}[1]{\mathbf{Mod}\text{-}#1}
\ncm{\Bimod}[2]{#1\text{-}\mathbf{Mod}\text{-}#2}
\ncm{\Comod}[1]{\mathbf{Comod}\text{-}#1}
\ncm{\Sh}{\mathbf{Sh}}
\ncm{\rMonCAT}{\textsf{r-MonCAT}}

\ncm{\A}{\mathcal{A}}
\ncm{\B}{\mathcal{B}}
\ncm{\C}{\mathcal{C}}
\ncm{\D}{\mathcal{D}}
\ncm{\E}{\mathcal{E}}
\ncm{\F}{\mathcal{F}}
\ncm{\G}{\mathcal{G}}
\ncm{\Ha}{\mathcal{H}}
\ncm{\I}{\mathcal{I}}
\ncm{\J}{\mathcal{J}}
\ncm{\K}{\mathcal{K}}
\ncm{\Ll}{\mathcal{L}}
\ncm{\M}{\mathcal{M}}
\ncm{\N}{\mathcal{N}}
\ncm{\Ou}{\mathcal{O}}
\ncm{\Pee}{\mathcal{P}}
\ncm{\R}{\mathcal{R}}
\ncm{\X}{\mathcal{X}}
\ncm{\W}{\mathcal{W}}
\ncm{\V}{\mathcal{V}}
\ncm{\U}{\mathcal{U}}
\ncm{\T}{\mathcal{T}}
\ncm{\Teven}{\mathcal{T}_\mathsf{e}}
\ncm{\Tcan}{\mathcal{T}_\mathsf{can}}

\ncm{\dom}{\operatorname{dom}}
\ncm{\cod}{\operatorname{cod}}
\ncm{\End}{\operatorname{End}}
\ncm{\Aut}{\operatorname{Aut}}
\ncm{\Hom}{\operatorname{Hom}}
\ncm{\kernel}{\operatorname{ker}}
\ncm{\Ker}{\operatorname{Ker}}
\ncm{\coker}{\operatorname{coker}}
\ncm{\Coker}{\operatorname{Coker}}
\ncm{\im}{\operatorname{im}}
\ncm{\Img}{\operatorname{Im}}
\ncm{\coim}{\operatorname{coim}}
\ncm{\id}{\operatorname{id}}
\ncm{\Center}{\operatorname{Center}}
\ncm{\colim}{\operatorname{colim}}
\ncm{\Colim}[1]{\underset{#1}{\operatorname{colim}}}
\ncm{\Lan}{\operatorname{Lan}}
\ncm{\Cone}{\operatorname{Cone}}
\ncm{\ev}{\operatorname{ev}}
\ncm{\coev}{\operatorname{coev}}
\ncm{\cf}{\operatorname{cf}}
\ncm{\hgt}{\operatorname{ht}}

\ncm{\ci}{\,\circ\,}
\ncm{\smp}{\ast}
\ncm{\smpq}{\smp_q}
\ncm{\bu}{\bullet}
\ncm{\bo}{\,\Box\,}
\ncm{\ot}{\otimes}
\ncm{\oV}{\odot}
\ncm{\oE}{\underset{\scriptscriptstyle E}{\odot}}
\ncm{\hot}{\,\bar{\ot}\,}
\ncm{\uot}{\,\Eob{\ot}\,} 
\ncm{\uoV}[1]{\underset{#1}\oV}
\ncm{\x}{\times}
\ncm{\ex}[1]{\underset{\scriptstyle #1}{\times}}
\ncm{\am}[1]{\underset{\scriptscriptstyle #1}{\ot}}
\ncm{\amo}[1]{\underset{\scriptstyle #1}{\ot}}
\ncm{\mash}{\Pisymbol{psy}{35}}
\ncm{\mashed}[1]{\underset{\scriptscriptstyle #1}{\Pisymbol{psy}{35}}}
\ncm{\cross}[1]{\underset{\scriptstyle #1}{\rtimes}}
\ncm{\rarr}[1]{\stackrel{#1}{\longrightarrow}}
\ncm{\larr}[1]{\stackrel{#1}{\longleftarrow}}
\ncm{\mapsot}{\leftarrow\!\!\!\raisebox{1pt}{$\scriptscriptstyle |$}}
\ncm{\oR}{\am{R}}
\ncm{\oS}{\am{S}}
\ncm{\cop}{\Delta}
\ncm{\oneT}{^{(1)}}
\ncm{\twoT}{^{(2)}}
\ncm{\threeT}{^{(3)}}
\ncm{\oneB}{_{(1)}}
\ncm{\twoB}{_{(2)}}
\ncm{\threeB}{_{(3)}}
\ncm{\eps}{\varepsilon}
\ncm{\bra}{\langle}
\ncm{\ket}{\rangle}
\ncm{\dirim}[1]{{#1}_*}
\ncm{\invim}[1]{{#1}^*}

\ncm{\asso}{a} 
\ncm{\luni}{l} 
\ncm{\runi}{r} 
\ncm{\iso}{\stackrel{\sim}{\rightarrow}}
\ncm{\iiso}{\rarr{\sim}}
\ncm{\ract}{\,\triangleleft\,}
\ncm{\lact}{\triangleright}
\ncm{\under}{\mbox{\,\rm\_}\,}
\ncm{\adj}{\dashv}
\ncm{\adjoint}{\dashv}
\ncm{\into}{\hookrightarrow}

\ncm{\can}{\mathrm{can}}
\ncm{\fgp}{\mathrm{fgp}}
\ncm{\op}{\mathrm{op}}
\ncm{\coop}{\mathrm{coop}}
\ncm{\rev}{\mathrm{rev}}
\ncm{\fl}{\mathrm{flat}}
\ncm{\sst}{\scriptstyle}
\ncm{\ssst}{\scriptscriptstyle}

\ncm{\eqby}[1]{\stackrel{(\ref{#1})}{=}}
\ncm{\lef}{{\ssst <}}
\ncm{\righ}{{\ssst >}}

\ncm{\NN}{\mathbb{N}}
\ncm{\ZZ}{\mathbb{Z}}
\ncm{\QQ}{\mathbb{Q}}
\ncm{\GG}{\mathbf{G}}
\ncm{\FF}{\mathbb{F}}
\ncm{\TT}{\mathsf{T}}
\ncm{\Q}{\mathsf{Q}}

\ncm{\pisharp}{\Pisymbol{psy}{35}}

\ncm{\Pre}{\hat\C}
\ncm{\She}{\mathsf{Sh}}
\ncm{\Sig}{{\Sigma}}
\ncm{\icog}{J}

\ncm{\alg}[1]{{\bar #1}}  
\ncm{\Eob}[1]{\underline{#1}} 
\ncm{\lf}[1]{\bar{#1}}
\ncm{\Forg}{\mathcal{G}}
\ncm{\JJ}{\mathcal{J}}

\ncm{\parallelpair}{
\parbox{43pt}{
\begin{picture}(43,8)
\put(3,6){\vector(1,0){37}}
\put(3,2){\vector(1,0){37}}
\end{picture}
}}
\ncm{\pair}[2]{\overset{#1}{\underset{#2}{\parallelpair}}}

\ncm{\longparallelpair}{
\parbox{63pt}{
\begin{picture}(63,8)
\put(3,6){\vector(1,0){57}}
\put(3,2){\vector(1,0){57}}
\end{picture}
}}
\ncm{\longpair}[2]{\overset{#1}{\underset{#2}{\longparallelpair}}}

\ncm{\longerparallelpair}{
\parbox{83pt}{
\begin{picture}(83,8)
\put(3,6){\vector(1,0){77}}
\put(3,2){\vector(1,0){77}}
\end{picture}
}}
\ncm{\longerpair}[2]{\overset{#1}{\underset{#2}{\longerparallelpair}}}

\ncm{\longrightarrowtail}{
\parbox{40pt}{
\begin{picture}(40,8)
\put(6,4){\line(-1,1){1.5}}
\put(6,4){\line(-1,-1){1.5}}
\put(6,4){\vector(1,0){31}}
\end{picture}
}}

\ncm{\longerrightarrowtail}{
\parbox{60pt}{
\begin{picture}(60,8)
\put(6,4){\line(-1,1){1.5}}
\put(6,4){\line(-1,-1){1.5}}
\put(6,4){\vector(1,0){51}}
\end{picture}
}}

\ncm{\longrarr}[1]{
\overset{#1}{
\parbox{40pt}{
\begin{picture}(40,8)
\put(3,4){\vector(1,0){34}}
\end{picture}
}}}

\ncm{\longerrarr}[1]{
\overset{#1}{
\parbox{70pt}{
\begin{picture}(70,8)
\put(3,4){\vector(1,0){64}}
\end{picture}
}}}

\ncm{\longerlarr}[1]{
\overset{#1}{
\parbox{70pt}{
\begin{picture}(70,8)
\put(67,4){\vector(-1,0){64}}
\end{picture}
}}}

\ncm{\longlarr}[1]{\overset{#1}{\longleftarrow}}

\ncm{\antiparallelpair}{
\parbox{23pt}{
\begin{picture}(23,4)
\put(3,3){\vector(1,0){17}}
\put(20,1){\vector(-1,0){17}}
\end{picture}
}}

\ncm{\invantiparallelpair}{
\parbox{23pt}{
\begin{picture}(23,4)
\put(3,1){\vector(1,0){17}}
\put(20,3){\vector(-1,0){17}}
\end{picture}
}}

\ncm{\dualpair}[2]{\overset{#1}{\underset{#2}{\antiparallelpair}}}

\ncm{\invdualpair}[2]{\overset{#1}{\underset{#2}{\invantiparallelpair}}}

\ncm{\coantiparallelpair}{
\parbox{23pt}{
\begin{picture}(23,4)
\put(3,1){\vector(1,0){17}}
\put(20,4){\vector(-1,0){17}}
\end{picture}
}}

\ncm{\codualpair}[2]{\overset{#1}{\underset{#2}{\coantiparallelpair}}}

\ncm{\binarydirectsum}[7]{#1\codualpair{#2}{#3}#4\dualpair{#5}{#6}#7}

\ncm{\equalizer}[1]{\overset{#1}{\longrightarrowtail}}

\ncm{\epi}[1]{\overset{#1}{\twoheadrightarrow}}

\ncm{\coequalizer}[1]{
\overset{#1}{
\parbox{40pt}{
\begin{picture}(40,8)
\put(2,4){\vector(1,0){32}} \put(37,4){\vector(1,0){0}}
\end{picture}
}}}

\ncm{\longcoequalizer}[1]{
\overset{#1}{
\parbox{60pt}{
\begin{picture}(60,8)
\put(2,4){\vector(1,0){52}} \put(57,4){\vector(1,0){0}}
\end{picture}
}}}

\ncm{\mono}[1]{\overset{#1}{\rightarrowtail}}

\ncm{\coequalizerfactorizationold}[9]{
\parbox[r]{115pt}{
\begin{picture}(115,70)(0,-5)
\put(0,48){$#1\longpair{#2}{#3}#4$} 
\end{picture}
}
\parbox[l]{80pt}{
\begin{picture}(80,70)(0,-5)
\put(2,51){\vector(1,0){70}} \put(75,51){\vector(1,0){0}}
\put(23,56){$\sst #5$}
\put(80,48){$#6$}
\put(2,47){\vector(2,-1){73}}
\put(48,30){$\sst #7$}
\dashline[+30]{3}(100,42)(100,12) \put(100,12){\vector(0,-1){0}}
\put(105,30){$\sst #8$}
\put(95,0){$#9$}
\end{picture}
}}

\ncm{\coequalizerfactorization}[9]{
\parbox[r]{143pt}{
\begin{picture}(143,70)(0,-5)
\put(0,48){$#1\longpair{#2}{#3}#4$} 
\end{picture}
}
\parbox[l]{80pt}{
\begin{picture}(80,70)(0,-5)
\put(2,51){\vector(1,0){70}} \put(75,51){\vector(1,0){0}}
\put(30,56){$\sst #5$}
\put(80,48){$#6$}
\put(2,47){\vector(2,-1){73}}
\put(41,30){$\sst #7$}
\dashline[+30]{3}(93,42)(93,12) \put(93,12){\vector(0,-1){0}}
\put(98,30){$\sst #8$}
\put(88,0){$#9$}
\end{picture}
}}

\begin{document}

\title[On the tensor product of modules over skew monoidal actegories]{On the tensor product of modules over\\ skew monoidal actegories}
\author{Korn\'el Szlach\'anyi}

\begin{abstract}
This paper is about skew monoidal tensored $\V$-categories (= skew monoidal hommed $\V$-actegories) and their categories of modules. A module over $\bra\M,\smp,R\ket$
is an algebra for the monad $T=R\smp\under$ on $\M$. We study in detail the skew monoidal structure of $\M^T$ and construct a skew monoidal forgetful functor $\M^T\to\,_E\M$
to the category of $E$-objects in $\M$ where $E=\M(R,R)$ is the endomorphism monoid of the unit object $R$. Then we give conditions for the forgetful functor to be strong monoidal and 
for the category $\M^T$ of modules to be monoidal. In formulating these conditions a notion of `self-cocomplete' subcategories of presheaves appears to be useful
which provides also some insight into the problem
of monoidality of the skew monoidal structures found by Altenkirch, Chapman and Uustalu on functor categories $[\C,\M]$.
\end{abstract}
\address{Wigner Research Centre for Physics of the Hungarian Academy of Sciences,\newline
\noindent 1525 Budapest, P.O.Box 49, Hungary}
\email{szlachanyi.kornel@wigner.mta.hu}
\maketitle

\section{Introduction}

A skew monoidal category consists of a category $\M$, a functor $\M\x\M\rarr{\smp} \M$, an object $R\in\M$ and comparison morphisms $\gamma:L\smp(M\smp N)\to
(L\smp M)\smp N$, $\eta:M\to R\smp M$, $\eps:M\smp R\to M$ satisfying usual monoidal category axioms but without the assumption that they are isomorphisms. This structure,  more primitive than a monoidal category, turned out to contain all algebraic information about a bialgebroid over a non-commutative ring $R$. As it was shown in 
\cite[Theorem 9.1]{SMC} the closed skew monoidal structures on $\Mod{R}$ with unit object being the right regular $R$-module are precisely the right bialgebroids $B$ 
over $R$. 
Under this correspondence the category of right $B$-modules becomes, not a `module category' over the skew monoidal $\M$ but simply, the Eilenberg-Moore category $\M^T$
of the monad $T=R\smp\under$ on $\M$. This canonical monad is present also in any monoidal category although it is an uninteresting identity monad. For skew monoidal
categories, however, it is the structure of $\M^T$ that embodies the representation theory of the skew monoidal $\M$ if viewed as a generalized bialgebroid.

It is well-known that the category of modules over a bialgebra, weak bialgebra or bialgebroid $B$ has a monoidal structure. It is defined using the coalgebra structure
of $B$ and the forgetful functor to the underlying monoidal category $\Mod{k}$ or $\Bimod{R}{R}$. In case of a skew monoidal $\M$ there is no a priori given underlying 
monoidal category, only the category $\M$ itself. This makes the construction of any (skew?) monoidal product on $\M^T$ non-trivial and in general impossible unless certain
right exactness conditions are fulfilled.
In the first version of this paper (presented in a talk \cite{Sz: Swansea}) we have constructed such a skew tensor product provided $\M$ had reflexive coequalizers and
$\smp$ preserved them. Independently, S. Lack and R. Street have given in \cite{Lack-Street: On monads and warpings} the same definition of the tensor but noticed also that 
it exists if $\smp$ preserves reflexive coequalizers only in the second argument. We call the tensor product $\hot$ on $\M^T$ the `horizontal tensor' in the hope that
the `vertical tensor' of comodules will also be found, although the latter so far has resisted all attempts.

Having defined the skew monoidal structure $\bra\M^T,\hot,\alg{R}\ket$ on the modules over $\M$ the next step is to construct a skew monoidal forgetful functor
to some underlying skew monoidal category. The natural candidate for the underlying category is the category $_E\M$ of $E$-objects in $\M$ where $E$ denotes the endomorphism 
monoid of $R$. In case of an $R$-bialgebroid $_E\M$ is just the bimodule category $\Bimod{R}{R}$. The skew monoidal product $\uot$ on $_E\M$ can be introduced
in two equivalent ways: Either as the horizontal tensor product of an even more primitive skew monoidal structure on $\M$ (which uses only the object $R$ but not $\smp$),
or as the Altenkirch-Chapman-Uustalu construction \cite{Altenkirch-Chapman-Uustalu} applied to $_E\M\equiv[E,\M]$ by considering $E$ as a 1-object category. 
The justification for this choice of $\uot$ is that a monadic and skew monoidal functor $\Forg:\M^T\to\,_E\M$ can be constructed.

We do all these constructions in the framework of tensored $\V$-categories. 
Assuming the existence of tensors $V\oV M$ of $V\in\V$ and $M\in\M$ is a technical assumption which, by the equivalence 
\cite{Gordon-Power,Janelidze-Kelly} between enriched categories \cite{Kelly} with tensors and actegories \cite{Pareigis II.} with hom objects, 
allows us to work with ordinary functors and natural transformations. 
This does not mean that skew monoidal categories without tensors are uninteresting. Skew monoidal monoids \cite{SMM} are good examples. 
Only in Section \ref{monoidality} we make an exception by studying general $\V$-categories $\M$ in order to characterize the functor categories $[\C,\M]$ for
which the skew monoidal structure of \cite{Altenkirch-Chapman-Uustalu} is monoidal. Introducing \textit{self-cocompleteness} we find that if $\M$ is a self-cocomplete subcategory of presheaves over $\C$, or, in the case of $\C=E$, it is a self-cocomplete subcategory $\W\subseteq\V_E$ of right $E$-modules then $[\C,\M]$, respectively $_E\M$, is a monoidal category.
Under appropriate right exactness conditions on $\bra\M,\smp,R\ket$ this implies that $\bra\M^T,\hot,\alg{R}\ket$ is also monoidal and the forgetful functor $\Forg$ is strong
monoidal.

\section{Preparations}

\subsection{Actegories}

Let $\bra \V,\oV,I,\asso,\luni,\runi\ket$ be a monoidal category in which the orientation of the coherence isomorphisms 
\begin{align*}
\asso_{U,V,W}&:U\oV(V\oV W)\iso(U\oV V)\oV W\\
\luni_V&:V\iso I\oV V\\
\runi_V&:V\oV I\iso V
\end{align*}
are chosen according to the right skew monoidal convention \cite{SMC}. Therefore the monoidal category axioms (the redundant set of 5 axioms \cite{CWM}), when written without 
inverses of $a$, $l$ or $r$, look precisely as the skew monoidal\footnote{Throughout this paper we use the term `skew monoidal' for `right skew monoidal'.} category axioms. 

\begin{defi} \label{def: V-act}
A left actegory over $\V$ consists of
\begin{itemize}
\item a category $\M$,
\item a functor $\under\oV\under:\V\x\M\to\M$
\item and natural isomorphisms 
\begin{align*}
\asso_{U,V,M}:U\oV(V\oV M)&\iso(U\oV V)\oV M\\
\luni_M:M&\iso I\oV M
\end{align*}
\end{itemize}
satisfying one pentagon and two triangle conditions
\begin{align}
(\asso_{U,V,W}\oV M)\ci\asso_{U,V\oV W,M}\ci(U\oV\asso_{V,W,M})&=\asso_{U\oV V,W,M}\ci\asso_{U,V,W\oV M}\\
\asso_{I,V,M}\ci\luni_{V\oV M}&=\luni_V\oV M\\
(\runi_V\oV M)\ci\asso_{V,I,M}\ci (V\oV\luni_M)&=V\oV M\,.
\end{align}
\end{defi}

For typographical reasons we are using the same symbol for the action of $\V$ on $\M$ and for the monoidal product of $\V$ and also the $\asso$ and $\luni$ have ambiguous meanings. Notice also the different form of the two triangle equations which is the consequence of our ``skew" convention of orienting left and right unitors. 
The advantage is that we can immediately say what a left skew actegory would be over a skew monoidal $\V$. It has the same structure and axioms as above, just the invertibility of all
coherence isomorphisms are relaxed. 

In fact we will need skew actegories in the sequel although not left but right ones. So we give here the right and lax version of Definition \ref{def: V-act}
(see also \cite{Lack-Street: lift})
\begin{defi} \label{def: skew-V-act}
A right skew actegory over a skew monoidal $\V$ consists of
\begin{itemize}
\item a category $\M$,
\item a functor $\under\oV\under:\M\x\V\to\M$
\item and natural transformations
\begin{align*}
\asso_{M,U,V}:M\oV(U\oV V)&\to(M\oV U)\oV V\\
\runi_M:M\oV I&\to M
\end{align*}
\end{itemize}
satisfying one pentagon and two triangle conditions
\begin{align}
\label{ract-1}
(\asso_{M,U,V}\oV W)\ci\asso_{M,U\oV V,W}\ci(M\oV\asso_{U,V,W})&=\asso_{M\oV U,V,W}\ci\asso_{M,U,V\oV W}\\
\label{ract-2}
(\runi_M\oV V)\ci\asso_{M,I,V}\ci (M\oV\luni_V)&=M\oV V\\
\label{ract-3}
\runi_{M\oV V}\ci\asso_{M,V,I}&=M\oV \runi_V\,.
\end{align}
\end{defi}
All skew monoidal categories in this paper are right skew monoidal. Thus both left and right 
actegories are considered over the same type of skew monoidal categories.
Whether left or right, the actegory axioms appear as subsets of the skew monoidal category axioms if we disregard the types 
($\M$ or $\V$) in the arguments of the natural transformations.

\begin{exa} \label{exa: M^T as M-actegory}
Let $\bra\M,\smp,R,\gamma,\eta,\eps\ket$ be any skew monoidal category. Then the Eilenberg-Moore category
$\M^T$ of the canonical monad $T=\bra R\smp\under,\mu,\eta\ket$ is a right skew actegory over $\M$. 
As a matter of fact, for a $T$-algebra $\alg{M}=\bra M,\alpha\ket$ and for an object $N$ of $\M$ we can define the $T$-algebra
\begin{equation}\label{eq: def ract}
\alg{M}\ract N\ :=\ \bra M\smp N,(\alpha\smp N)\ci\gamma_{R,M,N}\ket\,.
\end{equation}
The associator and the right unitor of $\M$ can be lifted to $\M^T$ as natural transformations
\begin{align}
\label{lifted gamma}
\gamma_{\alg{M},L,K}&:\alg{M}\ract(L\smp K)\to(\alg{M}\ract L)\ract K\\
\label{lifted eps}
\eps_{\alg{M}}&:\alg{M}\ract R\to\alg{M}\,.
\end{align}
They satisfy coherence conditions (\ref{ract-1}), (\ref{ract-2}) and (\ref{ract-3}) that make $\M^T$ a right $\M$-actegory in the skew sense.
\end{exa}

\begin{defi} \label{def: act mor}
A morphism $\M\to\N$ of left $\V$-actegories consists of
\begin{itemize}
\item a functor $F:\M\to\N$
\item and a natural transformation $F_{V,M}:V\oV FM\to F(V\oV M)$, called the strength, 
\end{itemize}
subject to the coherence conditions
\begin{align}
\label{strength-1}
F_{V\oV W,M}\ci\asso_{V,W,FM}&=F\asso_{V,W,M}\ci F_{V,W\oV M}\ci(V\oV F_{V,M})\\
\label{strength-2}
F_{I,M}\ci\luni_{FM}&=F\luni_M\,.
\end{align}
The morphism $F$ is called \textsl{strong} when its strength is an isomorphism and \textsl{strict} when its strength is the identity.
\end{defi}

\begin{defi} 
Let $F,G:\M\to\N$ be morphisms of left $\V$-actegories. A transformation $F\to G$ of actegory morphisms is nothing but a natural transformation $\nu:F\to G$ satisfying
\begin{equation} \label{ax: nu}
\nu_{V\oV M}\ci F_{V,M}\ =\ G_{V,M}\ci(V\oV\nu_M)\,.
\end{equation}
\end{defi}

The left $\V$-actegories, their morphisms and transformations are the objects, the 1-cells and 2-cells of a 2-category
 $\V$-$\ACT$. As an instance of doctrinal adjunction \cite{Kelly-doc-adj} a morphism $F:\M\to\N$ in $\V$-$\ACT$ has a right adjoint if and only if
it is strong and the ordinary functor $F:\M\to\N$ has a right adjoint.

In order to deal with (skew) monoidal products on $\V$-actegories as actegory morphisms the 2-category $\V$-$\ACT$ needs some
extra structure. In its stead we choose a multicategory strategy by introducing multivariable morphisms of actegories. From here on $\V$ is
symmetric monoidal.

\begin{defi}
For left $\V$-actegories $\M_1$,\dots,$\M_n$ and $\N$ an $n$-variable morphism $\M_1\x\dots\x\M_n\to\N$ of left $\V$-actegories is a functor $F:\M_1\x\dots\x\M_n\to\N$ such that
\begin{enumerate}
\item for each index $1\leq k\leq n$ there is a natural transformation
\[
F^{(k)}_{V,\underline{M}}\ :\ V\oV F(\underline{M})\to F(V\uoV{k}\underline{M}),
\]
where $V\uoV{k}\under$ denotes the action of $V$ on the $k$-th argument, such that for each fixed $M_i$, $i\neq k$, 
the $F^{(k)}_{V,\underline{M}}$ is a strength for the $k$-th partial functor
\[
F(M_1,\dots,M_{k-1},\under,M_{k+1},\dots, M_n)\ :\ \M_k\to\N\,;
\]
\item for each pair of indices $1\leq j < k\leq n$ the strengths $F^{(j)}$ and $F^{(k)}$ are compatible in the sense of 
\begin{align*}
F^{(k)}_{V,U\uoV{j}\underline{M}}\ci(V\oV F^{(j)}_{U,\underline{M}})&=
F^{(j)}_{U,V\uoV{k}\underline{M}}\ci(U\oV F^{(k)}_{V,\underline{M}})\ci\\
&\ci \asso^{-1}_{U,V,F\underline{M}}\ci(s_{V,U}\oV F\underline{M})\ci\asso_{V,U,F\underline{M}}
\end{align*}
where $s_{U,V}:U\oV V\iso V\oV U$ is the symmetry in $\V$.
\end{enumerate}

For a parallel pair of $n$-variable actegory morphisms $F,G:\M_1\x\dots\x\M_n\to\N$ a transformation $\nu:F\to G$
is a natural transformation that satisfies (\ref{ax: nu}) in each variable.
\end{defi}

In rare cases one needs a more general notion of actegory morphism than that of Definition \ref{def: act mor}.
\begin{defi} \label{def: gen act mor}
For left actegories $_\V\M$ and $_\W\N$ a morphism $_\V\M\to\,_\W\N$ is a pair $_FF'$ consisting of a monoidal functor $F:\V\to\W$ and of a functor $F':\M\to\N$ equipped with strength
\[
F'_{V,M}:FV\oV F'M\to F'(V\oV M)
\]
subject to the coherence conditions
\begin{align}
\label{gen-strength-1}
F'_{V_1\oV V_2,M}\ci(F_{V_1,V_2}\oV F'M)\ci\asso_{FV_1,FV_2,F'M}&=F'\asso_{V_1,V_2,M}\ci F'_{V_1,V_2\oV M}\ci(FV_1\oV F'_{V_2,M})\\
\label{gen-strength-2}
F'_{I,M}\ci(F_0\oV F'M)\ci\luni_{F'M}&=F'\luni_M\,.
\end{align}
In other words, $F'$ is a morphism of $\M\to F^*\N$ in $\V$-$\ACT$ where $F^*\N$ is obtained from $\N$ by ``restriction of scalars".

For a parallel pair $_FF'$, $_GG'$ of actegory morphisms $_\V\M\to_\W\N$ a transformation $_FF'\to\,_GG'$ is a pair $_\nu\nu'$ consisting of a monoidal natural transformation $\nu:F\to G$ and of a natural transformation $\nu':F'\to G'$ satisfying
\[
\nu'_{V\oV M}\ci F'_{V,M}\ =\ G'_{V,M}\ci (\nu_V\oV\nu'_M)\,.
\]
\end{defi}
The actegories with these generalized morphisms and transformations constitute a 2-category $\ACT$ which contains, as sub-2-categories, both the 2-category of monoidal categories and the 2-category $\V$-$\ACT$ for each monoidal $\V$.

\subsection{Enriched categories}

Heretofore the base category $\V$ was an arbitrary symmetric monoidal category. For this Subsection and for the rest of the paper $\V$ will be a 
symmetric closed monoidal category with all small limits and colimits. As for the $\V$-actegories we restrict attention to those left $\V$-actegories $\M_0$ for which the action, as a functor of the first argument, $\under\oV M:\V\to\M_0$, has a right adjoint $\M(M,\under)$. That such $\V$-actegories are equivalent to $\V$-enriched categories is in the folklore for a long time. A precise formulation of the equivalence has been given by Gordon and Power in \cite{Gordon-Power} for the general case of
action by enrichment over a bicategory. For a more recent account on the subject see \cite{Janelidze-Kelly}.

Let us summarize briefly how the $\V$-category structure emerges from the actegory structure. First one introduces the $\V$-valued hom $\M(M,N)$ as the value of the right adjoint $\M(M,\under)$ on $N$. Then the unit and counit of the adjunction
\begin{equation}
\label{tensor-adj0}
\M_0(V\oV M,N)\ \cong\ \V(V,\M(M,N))
\end{equation}
denoted by
\begin{align*}
\coev_{M,V}&:V\to\M(M,V\oV M)\\
\ev_{M,N}&:\M(M,N)\oV M\to N
\end{align*}
can be used to construct the strength of $\M(M,\under)$ by
\begin{align}
\label{def of chi}
\chi_{V,M,N}&:=\M(M,V\oV\ev_{M,N})\ci\M(M,\asso^{-1}_{V,\M(M,N),M})\ci\coev_{M,V\oV \M(M,N)}\\
\notag&:\ V\oV\M(M,N)\to\M(M,V\oV N)\,.
\end{align}
Equipped with this strength the functor $\M(M,\under)$ becomes a morphism of $\V$-actegories $\M_0\to\V$ and $\ev_M$
and $\coev_M$ become transformations of actegory morphisms, i.e., they satisfy
\begin{align}
\label{V-nat of ev}
V\oV\ev_{M,N}&=\ev_{M,V\oV N}\ci(\chi_{V,M,N}\oV M)\ci\asso_{V,\M(M,N),M}\\
\label{V-nat of coev}
\coev_{M,U\oV V}&=\M(M,\asso_{U,V,M})\ci\chi_{U,M,V\oV M}\ci(U\oV\coev_{M,V})\,.
\end{align}
Using these identities and the coherence conditions for the strength $\chi$ it not difficult to show that the $\V$-category composition and unit defined by
\begin{align}
\label{c via ev and H}
c_{L,M,N}&:=\M(L,\ev_{M,N})\ci\chi_{\M(M,N),L,M}\ :\ \M(M,N)\oV\M(L,M)\to\M(L,N)\\
i_M&:=\M(M,\luni^{-1}_M)\ci\coev_{M,I}\quad:\ \ I\ \to\ \M(M,M)
\end{align}
obey the axioms of a $\V$-enriched category. This $\V$-category $\M$ is then automatically tensored since the adjunction
(\ref{tensor-adj0}) readily implies the isomorphism
\[
\M(V\oV M,N)\cong [V,\M(M,N)]\,,
\]
$\V$-natural in $N$ where $[\under,\under]$ denotes left internal hom in $\V$.

This correspondence between hommed $\V$-actegories and tensored $\V$-categories extends to the morphisms and to the 
transformations in the following way.
\begin{pro}
Given tensored $\V$-categories $\M$ and $\N$ and given an object map $M\mapsto FM$ there is a bijection between
\begin{itemize}
\item the `arrow map' data $F_{M,M'}:\M(M,M')\to\N(FM,FM')$ of a $\V$-functor $F:\M\to\N$
\item and the strength data $F_{V,M}:V\oV FM\to F(V\oV M)$ of an actegory morphism $F:\M_0\to\N_0$
\end{itemize}
given by the following relations: 
\begin{align}
F_{V,M}&=\hat F_{M,V\oV M}\ci(\coev_{M,V}\oV FM)\\
\label{arrow map from strength}
F_{M,M'}&=\N(FM,\hat F_{M,M'})\ci\coev_{FM,\M(M,M')}
\end{align}
where $\hat F_{M,M'}:\M(M,M')\oV FM\to FM'$ is the diagonal arrow in the commuting square
\begin{equation}
\begin{CD}
\M(M,M')\oV FM@>F_{M,M'}\oV FM>>\N(FM,FM')\oV FM\\
@V{F_{\M(M,M'), M}}VV @VV{\ev_{FM,FM'}}V\\
F(\M(M,M')\oV M)@>F\ev_{M,M'}>>FM'\ .
\end{CD}
\end{equation}
Given $\V$-functors $F,G:\M\to\N$ between tensored $\V$-categories a natural transformation $\nu:F\to G$ is $\V$-natural
if and only if it is a transformation of actegory morphisms.
\end{pro}

Applying this Proposition to the partial functors of a multivariable actegory morphism $F$ we obtain $\V$-functors
$F_k:\M_k\to\N$ indexed by variables $M_i\in\M_i$, $i\neq k$, such that each pair $\bra F_j, F_k\ket$ obeys a compatibility condition of the form \cite[(4.4)]{Eilenberg-Kelly} and therefore determines a unique $\V$-functor 
$F:\M_1\ot\dots\ot\M_n\to\N$ on the tensor product of $\V$-categories. Note, however, that $\M_1\ot\dots\ot\M_n$
is not a tensored $\V$-category in a natural way so there is no strength for such a $\V$-functor. In its stead there are $n$ pieces of strengths, one for each partial $\V$-functor.

\subsection{Colimits} \label{ss: colim}

For a tensored $\V$-category $\M$ the general notion of $\V$-colimit \cite{Kelly} reduces to computation of tensors and enriched coends: For $\V$-functors $F:\D\to\M$ and $G:\D^\op\to\V$ the colimit of $F$ weighted by $G$ is the coend
\[
G\star F\ =\ \int^D\, GD\oV FD\,.
\]
But if we assume also that the $\V$-valued contravariant hom $\M(\under,M):\M_0^\op\to\V$ transforms colimits in $\M_0$ to limits in $\V$ then the above coend can be computed by an ordinary colimit in $\M_0$,
\[
\bigsqcup_{D,D'}\D(D,D')\oV (GD'\oV FD)\pair{}{}\bigsqcup_D \ GD\oV FD\coequalizer{}G\star F\,.
\]
\begin{lem}
Let $\M$ be a tensored $\V$-category and $\D$ an ordinary small category. Then $\M(\under,M)$ transforms colimits in $\M_0$ of functors $\D\to\M_0$ to limits in $\V$ for all $M\in\M_0$ if and only if $V\oV\under:\M_0\to\M_0$ preserves colimits of functors $\D\to\M_0$ for all $V\in\V$.
\end{lem}
\begin{proof}
This follows from the ordinary adjunction $\under\oV M\dashv\M(M,\under)$ and from the fact that the (set valued) hom-functors preserve (transform) and collectively reflect (co)limits.
\end{proof}
\begin{cor}
Let $\M_0$ be a left $\V$ actegory for which the functors $\under\oV M:\V\to\M_0$ have right adjoints for all $M\in\M_0$. Then the emerging tensored $\V$-category $\M$ is cocomplete if and only if
\begin{itemize}
\item the ordinary category $\M_0$ is cocomplete,
\item the functors $V\oV\under:\M_0\to\M_0$ preserve (ordinary) colimits.
\end{itemize}
\end{cor}
\begin{proof}
This is just a translation to the actegory language of \cite[Vol. 2, Theorem 6.6.14]{Borceux}.
\end{proof}

\section{Skew monoidal actegories}

Let $\bra \V,\oV,I,\asso,\luni,\runi,s\ket$ be a symmetric closed monoidal category with all small limits and colimits.
\begin{defi}
A skew monoidal $\V$-actegory consists of
\begin{itemize}
\item a $\V$-actegory $\M_0$ with action $\under\oV\under : \V\x\M_0\to\M_0$ and with coherence morphisms
\begin{align*}
\asso_{U,V,M}:U\oV(V\oV M)&\to(U\oV V)\oV M\\
\luni_M:M&\to I\oV M
\end{align*}
\item a 2-variable morphism $\under\smp\under:\M_0\x\M_0\to\M_0$ of $\V$-actegories, the skew tensor, with strengths
\begin{align*}
\Gamma_{V,M,N}&:V\oV(M\smp N)\to (V\oV M)\smp N\\
\Gamma'_{V,M,N}&:V\oV(M\smp N)\to M\smp(V\oV N)\,,
\end{align*}
\item an object $R$ of $\M_0$, called the skew unit,
\item and transformations of actegory morphisms
\begin{align*}
\gamma_{L,M,N}&:L\smp(M\smp N)\to(L\smp M)\smp N\\
\eta_M&:M\to R\smp M\\
\eps_M&:M\smp R\to M\,.
\end{align*}
\end{itemize}
The transformations are required to satisfy the skew monoidality axioms
\begin{align}
\label{SMC1}
(\gamma_{K,L,M}\smp N)\ci\gamma_{K,L\smp M,N}\ci(K\smp\gamma_{L,M,N})&=\gamma_{K\smp L,M,N}\ci\gamma_{K,L,M\smp N}\\
\label{SMC2}
\gamma_{R,M,N}\ci\eta_{M\smp N}&=\eta_M\smp N\\
\label{SMC3}
\eps_{M\smp N}\ci\gamma_{M,N,R}&=M\smp\eps_N\\
\label{SMC4}
(\eps_M\smp N)\ci\gamma_{M,R,N}\ci(M\smp\eta_N)&=M\smp N\\
\label{SMC5}
\eps_R\ci\eta_R&=R
\end{align}
\end{defi}

Unpacking the definition we obtain, beyond the 3 actegory axioms and the 5 skew monoidality axioms, 10 more:
\begin{align}
\label{sma-6}
\Gamma_{U\oV V,M,N}\ci\asso_{U,V,M\smp N}&=(\asso_{U,V,M}\smp N)\ci\Gamma_{U,V\oV M,N}\ci(U\oV\Gamma_{V,M,N})\\
\label{sma-7}
\Gamma_{I,M,N}\ci\luni_{M\smp N}&=\luni_M\smp N\\
\label{sma-8}
\Gamma'_{U\oV V,M,N}\ci\asso_{U,V,M\smp N}&=(M\smp\asso_{U,V,N})\ci\Gamma'_{U,M,V\oV N}\ci(U\oV\Gamma'_{V,M,N})\\
\label{sma-9}
\Gamma'_{I,M,N}\ci\luni_{M\smp N}&=M\smp\luni_N\\
\label{sma-10}
\Gamma'_{U,V\oV M,N}\ci (U\oV\Gamma_{V,M,N})&=\Gamma_{V,M,U\oV N}\ci(V\oV\Gamma'_{U,M,N})\ci\\
\notag
&\ci\asso^{-1}_{V,U,M\smp N}\ci(s_{U,V}\oV(M\smp N))\ci\asso_{U,V,M\smp N}\\
\label{sma-11}
\gamma_{V\oV L,M,N}\ci \Gamma_{V,L,M\smp N}&=
(\Gamma_{V,L,M}\smp N)\ci\Gamma_{V,L\smp M,N}\ci(V\oV\gamma_{L,M,N})\\
\label{sma-12}
\gamma_{L,V\oV M,N}\ci(L\smp\Gamma_{V,M,N})\ci\Gamma'_{V,L,M\smp N}&=
(\Gamma'_{V,L,M}\smp N)\ci\Gamma_{V,L\smp M,N}\ci(V\oV\gamma_{L,M,N})\\
\label{sma-13}
\gamma_{L,M,V\oV N}\ci(L\smp\Gamma'_{V,M,N})\ci\Gamma'_{V,L,M\smp N}&=
\Gamma'_{V,L\smp M,N}\ci(V\oV\gamma_{L,M,N})\\
\label{sma-14}
\Gamma'_{V,R,M}\ci(V\oV\eta_M)&=\eta_{V\oV M}\\
\label{sma-15}
\eps_{V\oV M}\ci\Gamma_{V,M,R}&=V\oV \eps_M\,.
\end{align}
The first five are the coherence conditions for the 2-variable actegory morphism $\bra \smp,\Gamma,\Gamma'\ket$ and the 
other five are the conditions for $\gamma$, $\eta$ and $\eps$ to be transformations of $\V$-actegory morphisms.

Of course, the above Definition is too general for our purposes. 
We are interested in $\V$-actegories the underlying category of which is a tensored $\V$-category and which satisfies also some right exactness conditions. 
\begin{defi} \label{def: rex smc}
Let $\V$ be a symmetric monoidal closed category with small limits and colimits.
A skew monoidal $\V$-actegory $\bra \M_0, \oV,\asso,\luni,\smp,\Gamma,\Gamma',R,\gamma,\eta,\eps\ket$ is called a 
\textsl{skew monoidal tensored $\V$-category} or, what is the same, a \textsl{skew monoidal hommed $\V$-actegory} if the following condition holds:
\begin{enumerate}
\item For each object $M$ of $\M_0$ the functor $\under\oV M:\V\to\M_0$ has a right adjoint $\M(M,\under)$.
\end{enumerate}
A skew monoidal tensored $\V$-category is called \textsl{r2-exact} if:
\begin{enumerate}
\setcounter{enumi}{1}
\item The category $\M_0$ has reflexive coequalizers.
\item For each object $V$ of $\V$ the functor $V\oV\under:\M_0\to\M_0$ preserves reflexive coequalizers.
\item For each object $M$ of $\M_0$ the functor $M\smp\under:\M_0\to\M_0$ is strong, i.e., $\Gamma'$ is an isomorphism.
\item For each object $M$ of $\M_0$ the functor $M\smp\under:\M_0\to\M_0$ preserves reflexive coequalizers.
\end{enumerate}
Similarly, we define \textsl{r1-exactness} by requiring, instead of (iv) and (v), invertibility of $\Gamma$ and preservation of reflexive coequalizers by the functors 
$\under\smp M$. If $\M$ is both r1-exact and r2-exact then it is called \textsl{r-exact}.
\end{defi}
\begin{lem} \label{lem: M^T is V-cat}
Let $\M$ be a r2-exact skew monoidal $\V$-category. Then the category $\M^T$ of modules over $\M$ is a tensored $\V$ category,
has reflexive coequalizers and the functors $V\oV\under:\M^T\to\M^T$ preserve them.
\end{lem}
\begin{proof}
The underlying endofunctor of $T$, $R\smp \under:\M_0\to\M_0$ is a strong morphism of actegories since $\Gamma'_{\under,R,\under}$ is invertible. The monad multiplication $\mu_M=(\eps_R\smp M)\ci\gamma_{R,R,M}$ and the unit $\eta_M$ are built from 2-cells of $\V$-$\ACT$
using composition and $\smp$, so they are 2-cells. too. 
The Eilenberg-Moore category of $T$ carries a $\V$-category structure such that the forgetful functor $G:\M^T\to\M$ is a strict morphism of actegories. For a $T$-algebra $\alg{M}=\bra M,\alpha\ket$ the action of $V\in\V$ has the form
\[
V\oV\alg{M}\ =\ \left\bra V\oV M, \ T(V\oV M)\longrarr{\Gamma^{\prime -1}_{V,R,M}}V\oV TM\longrarr{V\oV\alpha}V\oV M\right\ket\,.
\]
The right adjoint of $\under\oV \alg{M}$ is just the equalizer in $\V$
\[
\M^T(\alg{M},\alg{N})\equalizer{}\M(M,N)\longerpair{\M(TM,\beta)\ci T_{M,N}}{\M(\alpha,N)}\M(TM,N)
\]
for $T$-algebras $\alg{M}=\bra M,\alpha\ket$, $\alg{N}=\bra N,\beta\ket$. 
$\M^T$ has reflexive coequalizers since $T$ preserves and therefore the forgetful functor $G$ creates them.
Finally, $V\oV\under$ preserves reflexive coequalizers in $\M$ and commutes with $T$ via $\Gamma'$, therefore using preservation and creation of reflexive coequalizers by $G$ we conclude that $V\oV\under$ preserves reflexive coequalizers in $\M^T$.
\end{proof}
The above Lemma provides 3 of the 5 conditions for $\M^T$ to be also a r2-exact skew monoidal $\V$-category. Of course, the last 2 conditions of Definition \ref{def: rex smc} are so far meaningless since a skew monoidal product on $\M^T$ is not yet defined.
The next Section provides the missing skew tensor.

\section{Skew monoidal structure on $\M^T$} \label{sec: hot}

First we consider (ordinary) skew monoidal categories, forgetting the $\V$-structure for a while.
Let $\bra\M,\smp,R,\gamma,\eta,\eps\ket$ be a skew monoidal category with coequalizers of reflexive pairs and with $\smp$ preserving these coequalizers in the 2nd argument. This implies that $\M^T$ also has reflexive coequalizers and the forgetful functor $G:\M^T\to\M$ preserves them. 

Let $\mu_{N,M}:=(\eps_N\smp M)\ci\gamma_{N,R,M}$ as in \cite{SMC} which extends the monad multiplication: $\mu_M=\mu_{R,M}$.
For each object $N$ of $\M$ and for each $T$-algebra $\alg{M}=\bra M,\alpha\ket$ we choose a coequalizer
\begin{equation}\label{def of hot'}
N\smp TM\longpair{\mu_{N,M}}{N\smp \alpha}N\smp M\longcoequalizer{\pi_{N,\alg{M}}}N\hot\alg{M} 
\end{equation}
and call its colimit object the horizontal tensor product of $N$ and $\alg{M}$. Note that the parallel pair is reflexive:
$N\smp\eta_M$ is a common section.
After having defined $\hot$ on the objects we extend its definition to a functor $\M\x\M^T\to\M$ in such a way that $\pi$ becomes a natural transformation $\pi_{N,\alg{M}}:N\smp G\alg{M}\to N\hot\alg{M}$.

If $N$ is replaced by the underlying object of a $T$-algebra $\bra N,\beta\ket$ then $N\hot\alg{M}$ becomes the underlying object of a $T$-algebra $\alg{N}\hot\alg{M}$ the action $\psi$ of which is uniquely determined by the diagram
\begin{alignat}{2}
&R\smp(N\smp TM)\longpair{R\smp\mu_{N,M}}{R\smp(N\smp \alpha)}
&R\smp(N\smp M) \longcoequalizer{R\smp\pi_{N,\alg{M}}}
&R\smp(N\hot\alg{M}) \notag\\ 
\label{def of psi}
&\quad\parbox[c]{5pt}{\begin{picture}(5,40)\put(2,40){\vector(0,-1){40}}\end{picture}}\sst(\beta\smp TM)\ci\gamma_{R,N,TM}
&\ \parbox[c]{5pt}{\begin{picture}(5,40)\put(2,40){\vector(0,-1){40}}\end{picture}}\sst(\beta\smp M)\ci\gamma_{R,N,M}\qquad
&\quad\parbox[c]{5pt}{\begin{picture}(5,40)\dashline{3}(2,40)(2,2)\put(2,2){\vector(0,-1){0}}\end{picture}}\sst\psi\\
&N\smp TM\quad\quad\longpair{\mu_{N,M}}{N\smp \alpha}
&N\smp M \longcoequalizer{\pi_{N,\alg{M}}}
&N\hot\alg{M} \notag
\end{alignat}
where on the vertical arrows one can recognize the $T$-algebra actions of $\alg{N}\ract TM$ and $\alg{N}\ract M$,
as we defined them in Example \ref{exa: M^T as M-actegory}. Therefore $\pi_{N,\alg{M}}$ can be lifted to $\M^T$ as the coequalizer
\begin{equation}\label{def of hot}
\alg{N}\ract TM\longpair{\mu_{N,M}}{N\smp \alpha}
\alg{N}\ract M\longcoequalizer{\pi_{\alg{N},\alg{M}}}\alg{N}\hot\alg{M}
\end{equation}
where the new $\pi$ is related to the old one by
\[
G\pi_{\alg{N},\alg{M}}\ =\ \pi_{G\alg{N},\alg{M}}\,.
\]

\begin{lem}
For any $N$ and for any free $T$-algebra $\bra TM,\mu_M\ket$ the tensor product $N\hot\bra TM,\mu_M\ket$ as defined in
(\ref{def of hot'}) is isomorphic to the skew monoidal product $N\smp M$.
\end{lem}
\begin{proof}
The statement follows from the fact that
\begin{align}\label{eq: split coeq}
&N\smp T^2M\longpair{\mu_{N,TM}}{N\smp \mu_M}N\smp TM\longerrarr{\mu_{N,M}}N\smp M\\
\notag&\qquad\quad\ \ \longerlarr{N\smp T\eta_M}\qquad\quad\, \longerlarr{N\smp\eta_M}
\end{align}
is a split coequalizer in $\M$. 
\end{proof}
Let $F:\M\to\M^T$ denote the free algebra functor, $FM=\bra TM,\mu_M\ket$. Then it follows from the Lemma that there is a natural isomorphism
\begin{equation} \label{the j}
j_{N,M}:N\hot FM\iso N\smp M\qquad\text{such that}\qquad j_{N,M}\ci\pi_{N,TM}=\mu_{N,M}
\end{equation}
for each $N,M\in\M$. Using this isomorphism the coequalizer (\ref{def of hot'}) becomes 
\begin{equation}
N\hot FTM\longpair{N\hot\mu_M}{N\hot F\alpha}N\hot FM\longcoequalizer{N\hot\alpha}N\hot\alg{M} 
\end{equation}
which shows that (\ref{def of hot'}) is, up to isomorphism, the unique functor $\hot:\M\x\M^T\to\M$ which preserves reflexive coequalizers in the 2nd argument and for which
$N\hot FM\cong N\smp M$.

By the way,
the isomorphism $N\smp M\cong N\hot\bra TM,\mu_M\ket$ is in complete agreement with the formula \cite[Eqn. (30)]{SMC} for the skew monoidal product of a bialgebroid. 
But for a perfect analogy we should be able to write $\bra TM,\mu_M\ket$ as a `vertical tensor product' $M\ot H$ where $H=R\smp R$. 
The construction of such a vertical tensor product is still an open problem.

Although one would expect $\hot$ to be a monoidal product on the category $\M^T$ of modules, this is not the case. It is only skew monoidal although its left unitor is
always invertible as we will see soon.
Whether $\hot$ can be indeed monoidal will be investigated in Section \ref{monoidality}.

\begin{thm} \label{thm: hot}
Let $\bra\M,\smp,R,\gamma,\eta,\eps\ket$ be a skew monoidal category with coequalizers of reflexive pairs and with $\smp$ preserving these coequalizers in the second argument. Then $\M^T$, the Eilenberg-Moore category of the canonical monad $T$, has reflexive coequalizers and there exist functors
\[
\M\x\M^T\to\M\,,\qquad\M^T\x\M^T\to\M^T\,,
\]
both of them denoted by $\hot$ and called the horizontal tensor product, such that
\begin{enumerate}
\item $\M^T$ with $\hot$ as skew monoidal product is a skew monoidal category;
\item $\M$ with the action $\hot:\M\x\M^T\to\M$ is a right actegory over $\M^T$;
\item for each object $M\in\M$ the functor $M\hot\under$ and for each $T$-algebra $\alg{M}\in\M^T$ the functor $\alg{M}\hot\under$ preserves reflexive coequalizers;
\item the left unitor $\bar\eta$ of the skew monoidal product $\hot$ on $\M^T$ is invertible therefore the category 
of modules over $\M^T$ is trivial.
\end{enumerate}
\end{thm}
\begin{proof}
The action $\hot:\M\x\M^T\to\M$ is defined by the coequalizer (\ref{def of hot'}). 
Its associativity constraint can be obtained as follows. Let $N$ be an object in $\M$ and let $\alg{M}=\bra M,\alpha\ket$ and $\alg{L}=\bra L,\beta\ket$ be objects of $\M^T$. We need presentations for both the source and target of an arrow $N\hot(\alg{M}\hot\alg{L})\to(N\hot\alg{M})\hot\alg{L}$.
Consider the following two 3x3 diagrams with all columns, and in the case of the first also the rows, being reflexive coequalizers:
\begin{equation}\label{D1}
\parbox{300pt}{
\begin{picture}(300,170)(0,12)
\put(0,165){$N\smp T(M\smp TL)$}
\put(80,166){\vector(1,0){40}} \put(80,155){$\sst N\smp T(M\smp\beta)$}
\put(80,168){\vector(1,0){40}} \put(83,175){$\sst N\smp T\mu_{M,L}$}
\put(130,165){$N\smp T(M\smp L)$}
\put(200,167){\vector(1,0){40}} \put(202,174){$\sst N\smp T\pi_{M,\alg{L}}$}
\put(250,165){$N\smp T(M\hot\alg{L})$}

\put(38,150){\vector(0,-1){40}} \put(28,130){$\sst\mu$}
\put(41,150){\vector(0,-1){40}} \put(45,130){$\sst 1\smp[(\alpha\smp 1)\ci\gamma]$}
\put(158,150){\vector(0,-1){40}} \put(148,130){$\sst\mu$}
\put(161,150){\vector(0,-1){40}} \put(165,130){$\sst 1\smp[(\alpha\smp 1)\ci\gamma]$}
\put(278,150){\vector(0,-1){40}} \put(268,130){$\sst\mu$}
\put(281,150){\vector(0,-1){40}} \put(285,130){$\sst N\smp \psi$}

\put(5,90){$N\smp(M\smp TL)$}
\put(80,91){\vector(1,0){40}} \put(81,80){$\sst N\smp (M\smp\beta)$}
\put(80,93){\vector(1,0){40}} \put(83,100){$\sst N\smp \mu_{M,L}$}
\put(130,90){$N\smp (M\smp L)$}
\put(200,92){\vector(1,0){40}} \put(202,99){$\sst N\smp \pi_{M,\alg{L}}$}
\put(250,90){$N\smp (M\hot\alg{L})$}

\put(39,73){\vector(0,-1){30}} \put(43,58){$\sst\pi$}
\put(159,73){\vector(0,-1){30}} \put(163,58){$\sst\pi$}
\put(279,73){\vector(0,-1){30}} \put(283,58){$\sst\pi$}

\put(0,22){$N\hot(\alg{M}\ract TL)$}
\put(80,23){\vector(1,0){40}} \put(81,12){$\sst N\hot (M\smp\beta)$}
\put(80,25){\vector(1,0){40}} \put(83,32){$\sst N\hot \mu_{M,L}$}
\put(130,22){$N\hot (\alg{M}\ract L)$}
\put(200,24){\vector(1,0){40}} \put(202,31){$\sst N\hot \pi_{M,\alg{L}}$}
\put(250,22){$N\hot (\alg{M}\hot\alg{L})$}

\end{picture}}
\end{equation}
and
\begin{equation} \label{D2}
\parbox{300pt}{
\begin{picture}(300,170)(0,12)
\put(0,165){$(N\smp TM)\smp TL$}
\put(80,166){\vector(1,0){40}} \put(80,155){$\sst (N\smp\alpha) \smp TL$}
\put(80,168){\vector(1,0){40}} \put(83,175){$\sst \mu_{N,M} \smp TL$}
\put(130,165){$(N\smp M)\smp TL$}
\put(200,167){\vector(1,0){40}} \put(202,174){$\sst \pi_{N,\alg{M}} \smp TL$}
\put(250,165){$(N\hot\alg{M}) \smp TL$}

\put(38,150){\vector(0,-1){40}} \put(28,130){$\sst\mu$}
\put(41,150){\vector(0,-1){40}} \put(45,130){$\sst 1\smp\beta$}
\put(158,150){\vector(0,-1){40}} \put(148,130){$\sst\mu$}
\put(161,150){\vector(0,-1){40}} \put(165,130){$\sst 1\smp\beta$}
\put(278,150){\vector(0,-1){40}} \put(268,130){$\sst \mu$}
\put(281,150){\vector(0,-1){40}} \put(285,130){$\sst 1\smp\beta$}

\put(5,90){$(N\smp TM)\smp L$}
\put(80,91){\vector(1,0){40}} \put(81,80){$\sst (N\smp \alpha)\smp L$}
\put(80,93){\vector(1,0){40}} \put(83,100){$\sst \mu_{N,M}\smp L$}
\put(130,90){$(N\smp M)\smp L$}
\put(200,92){\vector(1,0){40}} \put(202,99){$\sst \pi_{N,\alg{M}}\smp L$}
\put(250,90){$(N\hot\alg{M}) \smp L$}

\put(39,73){\vector(0,-1){30}} \put(43,58){$\sst\pi$}
\put(159,73){\vector(0,-1){30}} \put(163,58){$\sst\pi$}
\put(279,73){\vector(0,-1){30}} \put(283,58){$\sst\pi$}

\put(0,22){$(N\smp TM)\hot\alg{L}$}
\put(80,23){\vector(1,0){40}} \put(81,12){$\sst (N\smp\alpha)\hot \alg{L}$}
\put(80,25){\vector(1,0){40}} \put(83,32){$\sst \mu_{N,M}\hot\alg{L}$}
\put(130,22){$(N\smp M)\hot \alg{L}$}
\put(200,24){\vector(1,0){40}} \put(202,31){$\sst \pi_{N,\alg{M}}\hot\alg{L}$}
\put(250,22){$(N\hot \alg{M})\hot\alg{L}$}

\end{picture}}
\end{equation}
Using the Diagonal Lemma \cite[p.4]{Johnstone-TT}  the diagonal of (\ref{D1}) is a coequalizer. This is the first row of the next diagram.  The diagonal of (\ref{D2}) is just a "commuting" fork. It appears in the second row of
\begin{equation}\label{def of bargamma}
\parbox{300pt}{
\begin{picture}(300,120)
\put(0,95){$N\smp T(M\smp TL)$}
\put(80,97){\vector(1,0){50}} \put(82,85){$\sst 1\smp[(\alpha\smp\beta)\ci\gamma]$}
\put(80,99){\vector(1,0){50}} \put(85,105){$\sst\mu\ci(1\smp T\mu)$}
\put(140,95){$N\smp(M\smp L)$}
\put(210,98){\vector(1,0){50}} \put(220,105){$\sst \pi\ci(1\smp\pi)$}
\put(270,95){$N\hot(\alg{M}\hot\alg{L})$}

\put(20,80){\vector(0,-1){50}} \put(25,54){$\sst\gamma\ci(1\smp\gamma)$}
\put(160,80){\vector(0,-1){50}} \put(165,54){$\sst\gamma$} 
\dashline{3}(290,80)(290,32)\put(290,32){\vector(0,-1){0}} \put(295,54){$\sst\bar\gamma$} 

\put(0,10){$(N\smp TM)\smp TL$}
\put(80,12){\vector(1,0){50}} \put(90,0){$\sst (1\smp\alpha)\smp\beta$}
\put(80,14){\vector(1,0){50}} \put(90,20){$\sst\mu\ci(\mu\smp 1)$}
\put(140,10){$(N\smp M)\smp L$}
\put(210,13){\vector(1,0){50}} \put(220,20){$\sst \pi\ci(\pi\smp 1)$}
\put(270,10){$(N\hot\alg{M})\hot\alg{L}$}

\end{picture}}
\end{equation}
with the unique arrow on the right hand side defining the associator $\bar\gamma_{N,\alg{M},\alg{L}}$
we have been looking for. In order to facilitate the proof of the pentagon relation for $\bar\gamma$ let us observe that the
left vertical arrow in (\ref{def of bargamma}) looks like an associator for a new skew monoidal product. This is almost the case.
\begin{lem} \label{lem: smp^2}
If $\bra\M,\smp,R,\gamma,\eta,\eps\ket$ is a skew monoidal category then the data
\begin{align}
M\smp^2 N&:=M\smp TN\\
R^2&:=R\\
\gamma^2_{L,M,N}&:=\gamma_{L,TM,TN}\ci(L\smp\gamma_{R,M,TN})\\
\eta^2_M&:=\eta_{TM}\ci\eta_M\\
\eps^2_M&:=\eps_M\ci(M\smp\eps_R)
\end{align}
obey 4 of the 5 skew monoidal category axioms: The failing axiom is the $\eps\gamma\eta$-triangle (\ref{SMC4}). But if $\eta$ is epi then 
$\bra\M,\smp^2,R^2,\gamma^2,\eta^2,\eps^2\ket$ is a skew monoidal category.
\end{lem}
Continuing with the proof of the Theorem we need only one consequence of the Lemma, namely that $\gamma^2$ satisfies the pentagon equation. Since $\hot$ preserves reflexive coequalizers it is now easy to show that $\bar\gamma$, too, satisfies the pentagon equation
\begin{equation}\label{hot pentagon}
(\bar\gamma_{N,\alg{M},\alg{L}}\hot\alg{K})\ci\bar\gamma_{N,\alg{M}\hot\alg{L},\alg{K}}\ci
(N\hot\bar\gamma_{\alg{M},\alg{L},\alg{K}})=
\bar\gamma_{N\hot\alg{M},\alg{L},\alg{K}}\ci\bar\gamma_{N,\alg{M},\alg{L}\hot\alg{K}}
\end{equation}
where $\bar\gamma_{\alg{M},\alg{L},\alg{K}}:\alg{M}\hot(\alg{L}\hot\alg{K})\to(\alg{M}\hot\alg{L})\hot\alg{K}$
is the lifting of $\bar\gamma_{M,\alg{L},\alg{K}}$ to $\M^T$. Such a lifting exists since not only $\mu_{N,M}$, $N\smp\alpha$
and $\pi_{N,\alg{M}}$ can be lifted by (\ref{def of hot}) but also $\gamma$ and $\gamma^2$ by (\ref{lifted gamma}).
This lifted $\bar\gamma$ is nothing but the associator for $\hot$ as a skew monoidal product. 

The next task is to study the unitors and their coherence conditions.
As for the unit object for $\hot$ we define it to be the $T$-algebra $\alg{R}=\bra R,\eps_R\ket$. 

The right unitor $\bar\eps_N:N\hot\alg{R}\to N$ can be obtained from unique factorization of $\eps_N$ through the coequalizer $\pi_{N,\alg{R}}$
\begin{equation}\label{def of bareps}
\bar\eps_N\ci\pi_{N,\alg{R}}\ =\ \eps_N\,,\qquad N\in\ob\M
\end{equation}
since $\eps_N\ci\mu_{N,R}=\eps_N\ci(N\smp\eps_R)$ is an identity in any skew monoidal category.
For $\alg{N}=\bra N,\alpha\ket$ we can define $\bar\eps_{\alg{N}}:\alg{N}\hot\alg{R}\to \alg{N}$ simply as the lift to $\M^T$
of $\bar\eps_N$ since in the diagram
\[
\begin{CD}
R\smp(N\smp R)@>T\pi>>R\smp(N\hot \alg{R})@>T\bar\eps_N>>R\smp N\\
@V{(\alpha\smp 1)\ci\gamma}VV @VV{\psi}V @VV{\alpha}V\\
N\smp R@>\pi>>N\hot \alg{R}@>\bar\eps_N>> N
\end{CD}
\]
the left square is commutative because of (\ref{def of psi}), the outer rectangle is commutative because
$\eps_N\ci(\alpha\smp R)\ci\gamma_{R,N,R}=\alpha\ci\eps_{TN}\ci\gamma_{R,N,R}=\alpha\ci T\eps_N$ is an identity
and $R\smp\pi_{N,\alg{R}}$ is epimorphic; so commutativity of the right square follows.

In order to prove the triangle relation (\ref{SMC3}) for $\hot$ we substitute $\alg{L}=\alg{R}$ in (\ref{def of bargamma}) and
compose the second row with $\bar\eps$. The calculation
\begin{align*}
&\bar\eps_{N\hot\alg{M}}\ci\bar\gamma_{N,\alg{M},\alg{R}}\ci\pi_{N,\alg{M}\hot\alg{R}}\ci(N\smp\pi_{M,\alg{R}})=\\
&\bar\eps_{N\hot\alg{M}}\ci\pi_{N\hot\alg{M},\alg{R}}\ci(\pi_{N,\alg{M}}\smp R)\ci\gamma_{N,M,R}=\\
&\eps_{N\hot\alg{M}}\ci(\pi_{N,\alg{M}}\smp R)\ci\gamma_{N,M,R}=\\
&\pi_{N,\alg{M}}\ci\eps_{N\smp M}\ci\gamma_{N,M,R}=\\
&\pi_{N,\alg{M}}\ci(N\smp\eps_M)=\\
&\pi_{N,\alg{M}}\ci(N\smp\bar\eps_M)\ci(N\smp\pi_{M,\alg{R}})=\\
&(N\hot\bar\eps_{\alg{M}})\ci\pi_{N,\alg{M}\hot\alg{R}}\ci(N\smp\pi_{M,\alg{R}})
\end{align*}
proves the coherence condition
\begin{equation}\label{hot' SMC3}
\bar\eps_{N\hot\alg{M}}\ci\bar\gamma_{N,\alg{M},\alg{R}}\ =\ N\hot\bar\eps_{\alg{M}}\,.
\end{equation}
Lifting to $\M^T$ we obtain for all $T$-algebras $\alg{N}$ and $\alg{M}$ that
\begin{equation}\label{hot SMC3}
\bar\eps_{\alg{N}\hot\alg{M}}\ci\bar\gamma_{\alg{N},\alg{M},\alg{R}}\ =\ \alg{N}\hot\bar\eps_{\alg{M}}\,.
\end{equation}

The left unitor $\bar\eta_{\alg{M}}:\alg{M}\to \alg{R}\hot\alg{M}$ is defined simply by the equation
\begin{equation}\label{def of bareta}
\eta_{\alg{M}}\ :=\ \pi_{R,\alg{M}}\ci\eta_M
\end{equation}
which happens to be a $T$-algebra morphism in spite of that $\eta_M$ is not. The outer rectangle of 
\[
\begin{CD}
TM@>T\eta_M>>T^2M@>T\pi_{R,\alg{M}}>>T(R\hot \alg{M})\\
@V{\alpha}V{\qquad\ \bullet}V @VV{\mu_M}V @VV{\psi}V\\
M@>\eta_M>>TM@>\pi_{R,\alg{M}}>>R\hot\alg{M}
\end{CD}
\]
commutes although the left square does not. 

The proof of the $\bar\eta$-triangle
\begin{equation}\label{hot SMC2}
\bar\gamma_{\alg{R},\alg{M},\alg{N}}\ci\bar\eta_{\alg{M}\hot\alg{N}}\ =\ \bar\eta_{\alg{M}}\hot\alg{N}
\end{equation}
goes as follows:
\begin{align*}
&\bar\gamma_{\alg{R},\alg{M},\alg{N}}\ci\bar\eta_{\alg{M}\hot\alg{N}}\ci\pi_{M,\alg{N}}=\\
&\bar\gamma_{\alg{R},\alg{M},\alg{N}}\ci\pi_{R,\alg{M}\hot\alg{N}}\ci\eta_{M\hot\alg{N}}\ci\pi_{M,\alg{N}}=\\
&\bar\gamma_{\alg{R},\alg{M},\alg{N}}\ci\pi_{R,\alg{M}\hot\alg{N}}\ci(R\smp\pi_{M,\alg{N}})\ci\eta_{M\smp N}
\eqby{def of bargamma}\\
&\pi_{R\ot\alg{M},\alg{N}}\ci(\pi_{R,\alg{N}}\smp N)\ci\gamma_{R,M,N}\ci\eta_{M\smp N}=\\
&\pi_{R\ot\alg{M},\alg{N}}\ci(\pi_{R,\alg{N}}\smp N)\ci(\eta_M\smp N)=\\
&\pi_{R\ot\alg{M},\alg{N}}\ci(\bar\eta_{\alg{M}}\smp N)=\\
&(\bar\eta_{\alg{M}}\hot\alg{N})\ci\pi_{M,\alg{N}}\,.
\end{align*}

The mixed triangle for $\hot$, i.e., the coherence condition
\begin{equation}\label{hot' SMC4}
(\bar\eps_N\hot \alg{M})\ci\bar\gamma_{N,\alg{R},\alg{M}}\ci(N\hot\bar\eta_{\alg{M}})\ =\ N\hot\alg{M}
\end{equation}
and its lifted version
\begin{equation}\label{hot SMC4}
(\bar\eps_{\alg{N}}\hot \alg{M})\ci\bar\gamma_{\alg{N},\alg{R},\alg{M}}\ci(\alg{N}\hot\bar\eta_{\alg{M}})\ =\ \alg{N}\hot\alg{M}
\end{equation}
can be obtained as follows:
\begin{align*}
&(\bar\eps_N\hot \alg{M})\ci\bar\gamma_{N,\alg{R},\alg{M}}\ci(N\hot\bar\eta_{\alg{M}})\ci\pi_{N,\alg{M}}=\\
&(\bar\eps_N\hot \alg{M})\ci\bar\gamma_{N,\alg{R},\alg{M}}\ci\pi_{N,\alg{R}\hot\alg{M}}\ci(N\smp\pi_{R,\alg{M}})
\ci(N\smp\eta_M)=\\
&(\bar\eps_N\hot \alg{M})\ci
\underset{(\pi_{N,\alg{R}}\hot\alg{M})\ci\pi_{N\smp R,\alg{M}}}
{\underbrace{\pi_{N\hot\alg{R},\alg{M}}\ci(\pi_{N,\alg{R}}\smp M)}}
\ci\gamma_{N,R,M}\ci(N\smp\eta_M)=\\
&(\eps_N\hot \alg{M})\ci\pi_{N\smp R,\alg{M}}\ci\gamma_{N,R,M}\ci(N\smp\eta_M)=\\
&\pi_{N,\alg{M}}\ci(\eps_N\smp M)\ci\gamma_{N,R,M}\ci(N\smp\eta_M)=\\
&\pi_{N,\alg{M}}\,.
\end{align*}

Finally, the axiom (5) for $\hot$, i.e., the relation
\begin{equation}\label{hot SMC5}
\bar\eps_{\alg{R}}\ci\bar\eta_{\alg{R}}\ =\ \alg{R}
\end{equation}
follows immediately from (\ref{def of bareta}), (\ref{def of bareps}) and (\ref{SMC5}).
In this way the data $\bra\M^T,\hot,\alg{R},\bar\gamma,\bar\eta,\bar\eps\ket$ is a skew monoidal category by the lifted version of (\ref{hot pentagon}) and by (\ref{hot SMC2}), (\ref{hot SMC3}), (\ref{hot SMC4}) and (\ref{hot SMC5}). This proves 
(i). The three axioms of Definition \ref{def: skew-V-act} for $\M$ to be a right $\M^T$-actegory are provided by
(\ref{hot pentagon}), (\ref{hot' SMC3}) and (\ref{hot' SMC4}). This proves (ii).

It should be clear from (\ref{def of hot'}) and (\ref{def of hot}), that $\hot$, both as a tensor product and as an action, preserves
all type of colimits that are preserved by $\smp$. Under the present assumptions this means that $\hot$ preserves at least the reflexive coequalizers in the 2nd argument. This proves (iii).

For any $T$-algebra $\alg{M}=\bra M,\alpha\ket$ the canonical presentation
\[
T^2 M\pair{\mu_M}{T\alpha}TM\coequalizer{\alpha}M
\]
and the definition of the tensor $R\hot\alg{M}$
\[
R\smp TM\pair{\mu_{R,M}}{R\smp\alpha}R\smp M\coequalizer{\pi_{R,\alg{M}}}R\hot\alg{M}
\]
are coequalizers of the same reflexive pair. Therefore there exists an isomorphism $i$ such that $i\ci\pi_{R,\alg{M}}=\alpha$.
By the definition of $\bar\eta$ in (\ref{def of bareta})
\[
i\ci\bar\eta_{\alg{M}}=i\ci\pi_{R,\alg{M}}\ci\eta_M=\alpha\ci\eta_M=M
\]
hence $\bar\eta_{\alg{M}}$ is an isomorphism. Now an algebra for the monad $\alg{R}\hot\under$ with underlying object,
let us say, $\alg{M}$ has unique action $\bar\eta^{-1}_{\alg{M}}:\alg{R}\hot\alg{M}\to\alg{M}$. This means that the category of
modules over $\M^T$ is $\M^T$ itself. This proves (iv).
\end{proof}

\begin{pro} \label{pro: G}
Under the assumptions of Theorem \ref{thm: hot} the forgetful functor $G:\M^T\to\M$, $\bra M,\alpha\ket\to M$,
is a strict morphism of right $\M^T$-actegories and has the structure of a strictly normal skew monoidal functor.
\end{pro}
\begin{proof}
The way we lifted $N\hot\alg{M}$ to $\M^T$ in (\ref{def of psi}) ensures that
\[
G(\alg{N}\hot\alg{M})\ =\ G\alg{N}\hot \alg{M}
\]
which is precisely the statement that the strength of $G$ as a right $\M^T$-actegory morphism is the identity.
But more interesting is that the coequalizer in (\ref{def of hot'}) can be read as a multiplicativity constraint for $G$,
\[
G_{\alg{N},\alg{M}}\ :=\ \pi_{G\alg{N},\alg{M}}\ :\ G\alg{N}\smp G\alg{M}\to G(\alg{N}\hot\alg{M})\,.
\]
The unit constraint is the identity $R=G\alg{R}$. The right square in (\ref{def of bargamma}) is just the hexagon expressing
associativity of $G_{\alg{N},\alg{M}}$ and the unitality tetragons reduce to triangles which are but the definitions 
(\ref{def of bareta}) and (\ref{def of bareps}) of $\bar\eta$ and $\bar\eps$, respectively.
\end{proof}

Combining the results of the above Theorem and Proposition with the $\V$-actegory structure we obtain the following.

\begin{thm} \label{thm: V-hot}
Let $\M$ be a r2-exact skew monoidal $\V$-category. Then $\M^T$ is also a r2-exact skew monoidal $\V$-category
and the forgetful functor $G:\M^T\to\M$ is a skew monoidal $\V$-functor which is strict as an actegory morphism.
\end{thm}
\begin{proof}
By Lemma \ref{lem: M^T is V-cat}, Theorem \ref{thm: hot} and Proposition \ref{pro: G} we only have to show that
\begin{enumerate}
\item $\hot$ is a $\V$-functor, that is to say a 2-variable actegory morphism,
\item the partial functor $\alg{M}\hot\under$ is strong;
\item $\bar\gamma$, $\bar\eta$ and $\bar\eps$ are $\V$-natural, i.e., they are transformations of actegory morphisms.
\end{enumerate}
(i) The two strengths of $\hot$ can be introduced as the unique arrows making the diagrams
\begin{alignat}{2}
&V\oV(N\smp TM)\quad\longpair{V\oV\mu_{N,M}}{V\oV(N\smp \alpha)}
&V\oV(N\smp M) \longcoequalizer{V\oV\pi_{N,\alg{M}}}
&V\oV(N\hot\alg{M}) \notag\\ 
\label{def of barGamma}
&\qquad\parbox[c]{5pt}{\begin{picture}(5,40)\put(2,40){\vector(0,-1){40}}\end{picture}}\sst\Gamma_{V,N,TM}
&\ \parbox[c]{5pt}{\begin{picture}(5,40)\put(2,40){\vector(0,-1){40}}\end{picture}}\sst\Gamma_{V,N,M}\qquad\qquad\qquad
&\qquad\parbox[c]{5pt}{\begin{picture}(5,40)\dashline{3}(2,40)(2,2)\put(2,2){\vector(0,-1){0}}\end{picture}}
\sst\bar\Gamma_{V,N,\alg{M}}\\
&(V\oV N)\smp TM\quad\longpair{\mu_{V\oV N,M}}{(V\oV N)\smp \alpha}
&(V\oV N)\smp M \longcoequalizer{\pi_{V\oV N,\alg{M}}}
&(V\oV N)\hot\alg{M} \notag
\end{alignat}
and
\begin{alignat}{2}
&V\oV(N\smp TM)\quad\longpair{V\oV\mu_{N,M}}{V\oV(N\smp \alpha)}
&V\oV(N\smp M) \longcoequalizer{V\oV\pi_{N,\alg{M}}}
&V\oV(N\hot\alg{M}) \notag\\ 
\label{def of barGamma'}
&\qquad\parbox[c]{5pt}{\begin{picture}(5,40)\put(2,40){\vector(0,-1){40}}\end{picture}}
\sst(N\smp\Gamma'_{V,R,M})\ci\Gamma'_{V,N,TM}
&\ \parbox[c]{5pt}{\begin{picture}(5,40)\put(2,40){\vector(0,-1){40}}\end{picture}}\sst\Gamma'_{V,N,M}\qquad\qquad\qquad
&\qquad\parbox[c]{5pt}{\begin{picture}(5,40)\dashline{3}(2,40)(2,2)\put(2,2){\vector(0,-1){0}}\end{picture}}
\sst\bar\Gamma'_{V,N,\alg{M}}\\
&N\smp T(V\oV M)\quad\longpair{\mu_{N,V\oV M}}{N\smp \hat\alpha}
&N\smp (V\oV M) \longcoequalizer{\pi_{N,V\oV\alg{M}}}
&N\hot(V\oV\alg{M}) \notag
\end{alignat}
commutative where in the second we used the notation $\hat\alpha:=(V\oV\alpha)\ci\Gamma^{\prime -1}_{V,R,M}$.
Calling the first vertical arrows of these diagrams $\Gamma^2_{V,N,M}$ and $\Gamma^{\prime 2}_{V,N,M}$, respectively,
we find, as in Lemma \ref{lem: smp^2}, that $\Gamma^2$ and $\Gamma^{\prime 2}$ obey the 5 coherence conditions making them strengths for the $\V$-functor $\smp^2$. It follows now from the exactness conditions that $\bar\Gamma$ and $\bar\Gamma'$ satisfy the coherence conditions making them strengths for $\hot$.

(ii) Since the first 2 vertical arrows in (\ref{def of barGamma'}) are isomorphisms, so is the third.

(iii) $\V$-naturality of $\bar\gamma$ means proving equations (\ref{sma-11}), (\ref{sma-12}) and (\ref{sma-13}) with $\gamma$, $\Gamma$ and $\smp$ replaced by $\bar\gamma$, $\bar\Gamma$ and $\hot$, respectively.
After composing these equations with the epimporphism $(V\oV\pi_{L,\alg{M}\hot\alg{N}})\ci(V\oV(L\smp\pi_{\alg{M},\alg{N}}))$,
using the relations (\ref{def of bargamma}), (\ref{def of barGamma}) or (\ref{def of barGamma'}) and then (\ref{sma-11}) or (\ref{sma-12}) or (\ref{sma-13}), respectively, the barred versions of (\ref{sma-11}), (\ref{sma-12}) or (\ref{sma-13}) can be proven. 
One can similarly verify the barred versions of (\ref{sma-14}) or (\ref{sma-15}) expressing $\V$-naturality of $\bar\eta$ and $\bar\eps$.
\end{proof}

\section{A lifting theorem}

Let $\M$ and $\M'$ be skew monoidal categories and let $K:\M\to\M'$ be a skew monoidal functor.
We assume that both $\M$ and $\M'$ have reflexive coequalizers and that both skew monoidal products preserve reflexive coequalizers in the second argument. Thus both module categories $\M^T$ and $\M^{\prime T'}$ have a skew monoidal structure by Theorem \ref{thm: hot} and have  skew monoidal forgetful functors $G$ and $G'$ as in the diagram
\begin{equation} \label{lift of K}
\begin{CD}
\M^T@>\lf{K}>?>\M^{\prime T'}\\
@VV{G}V @VV{G'}V\\
\M@>K>>\M'
\end{CD}
\end{equation}
We are asking whether a skew monoidal lifting $\lf{K}$ of $K$ exists that makes the diagram (strictly) commutative in the 2-category of skew monoidal categories. This problem is not simply a lifting problem in the 2-category of (r2-exact) skew monoidal categories because the canonical monad is not skew monoidal. (In fact $T$ is closer to be skew opmonoidal than skew monoidal but it is neither of them.) But we do not claim the existence of $\lf{K}$ for general monads, either. 

\begin{thm} \label{thm: lift}
Let $\M$ and $\M'$ be as above. Then any skew monoidal functor $K:\M\to\M'$ has a skew monoidal lift $\lf{K}$ to the 
categories of modules.

If $K$ is a skew monoidal $\V$-functor with $\M$, $\M'$ being r2-exact skew monoidal $\V$-categories
then the lift $\lf{K}$ is also a skew monoidal $\V$-functor.
\end{thm}
\begin{proof}
Merely as ordinary functors the existence of $\lf{K}$ follows from the universal property of the forgetful functor $G'$ as it was formulated in \cite[Theorem 8.1]{Lack-Street: On monads and warpings}. Indeed, $KG$ is a skew monoidal functor the domain of which is normal skew monoidal.
Hence it must factor uniquely through $G'$ as skew monoidal functors.
Nevertheless it is instructive to look at the form $\lf{K}$ takes due to the special form of $KG$.

For arbitrary monads $T$, $T'$ the possible lifts of $K$ are known to be in bijection with natural transformations $\kappa:T'K\to KT$ such that the
pair $\bra K,\kappa\ket$ forms a monad morphism \cite{Street: Mnd} $\bra\M,T\ket\to\bra\M',T'\ket$. Then the lift associated to $\kappa$ is given
by
\[
\lf{K}\ :\ \bra M,\alpha\ket\ \mapsto\ \bra KM,K\alpha\ci\kappa_M\ket\,.
\]
A candidate for a monad morphism is provided by the skew monoidal structure of $K$ \cite[Lemma 2.7]{SMC}, namely $\kappa_M=K_{R,M}\ci(K_0\smp KM)$. We claim that this choice does the job. 

The skew monoidal structure of $\lf{K}$ can be constructed as follows. The diagram valid for all $\alg{M_i}=\bra M_i,\alpha_i\ket\in\M^T$, $i=1,2$, 
\begin{alignat}{2} 
&KM_1\smp T'KM_2\ \ \longpair{\mu'_{KM_1,KM_2}}{KM_1\smp(K\alpha_2)\ci\kappa_{M_2}}
&KM_1\smp KM_2 \longcoequalizer{\pi'_{\lf{K}\alg{M_1},\lf{K}\alg{M_2}}}
&\lf{K}\alg{M_1}\hot \lf{K}\alg{M_2} \notag\\ 
\label{lift K2}
&\quad\parbox[c]{5pt}{\begin{picture}(5,40)\put(2,40){\vector(0,-1){40}}\end{picture}}\sst K_{M_1,TM_2}\ci(KM_1\smp\kappa_{M_2})
&\parbox[c]{5pt}{\begin{picture}(5,40)\put(2,40){\vector(0,-1){40}}\end{picture}}\sst K_{M_1,M_2}\qquad\qquad\quad
&\quad\quad\quad\parbox[c]{5pt}{\begin{picture}(5,40)\dashline{3}(2,40)(2,2)\put(2,2){\vector(0,-1){0}}\end{picture}}\sst \lf{K}_{\alg{M}_1,\alg{M}_2}\\
&K(M_1\smp TM_2)\quad\longpair{K\mu_{M_1,M_2}}{K(M_1\smp \alpha_2)}
&K(M_1\smp M_2)\quad \longrarr{K\pi_{\alg{M}_1,\alg{M}_2}}\quad
&\lf{K}(\alg{M}_1\hot\alg{M}_2) \notag
\end{alignat}
defines the components of $\lf{K}_2$.
The unit constraint for $\lf{K}$ is defined by
\begin{equation} \label{lift K0}
\lf{K}_0:=\left(\bra R',\eps'_{R'}\ket\rarr{K_0}\lf{K}\bra R,\eps_R\ket\right).
\end{equation}
We leave it to the reader to check that $\bra \lf{K},\lf{K}_2,\lf{K}_0\ket$ is a skew monoidal functor and makes (\ref{lift of K}) commutative. 

As far as the $\V$-structure is concerned, the assumption that $K$ is a skew monoidal $\V$-functor includes the assumption that the multiplicativity constraint $K_2$ is $\V$-natural in the second argument,
\begin{equation}\label{V-nat of K_2}
K_{L,V\oV M}\ci(KL\smp K_{V,M})\ci\Gamma'_{V,KL,KM}\ =\ K\Gamma'_{V,L,M}\ci K_{V,L\smp M}\ci (V\oV K_{L,M})\,.
\end{equation}
This relation for $L=R$ suffices to define the strength of $\lf{K}$. As a matter of fact,
\[
V\oV\lf{K}\alg{M}\longrarr{\lf{K}_{V,\alg{M}}}\lf{K}(V\oV\alg{M})
\]
can be introduced by 
\begin{equation} \label{def of strength of Kbar}
G'\lf{K}_{V,\alg{M}}:=K_{V,G\alg{M}}
\end{equation}
since $K_{V,M}$ is a well-defined arrow 
\[
\bra V\oV KM, (V\oV K\alpha)\ci(V\oV\kappa_M)\ci\Gamma^{\prime -1}_{V,R',KM}\ket\longrarr{K_{V,M}}
\bra K(V\oV M), K(V\oV\alpha)\ci K\Gamma^{\prime -1}_{V,R,M}\ci\kappa_{V\oV M}\ket
\]
in $\M^{\prime T'}$ due to commutativity of the diagram
\[
\begin{CD}
T'(V\oV KM)@<\Gamma'<\sim<V\oV T'KM@>V\oV \kappa_M>> V\oV KTM@>V\oV K\alpha>>V\oV KM\\
@V{T' K_{V,M}}VV @. @V{K_{V,TM}}VV @VV{K_{V,M}}V\\
T'K(V\oV M)@>\kappa_{V\oV M}>>KT(V\oV M)@<K\Gamma'<\sim<K(V\oV TM)@>K(V\oV\alpha)>>K(V\oV M)
\end{CD}
\]
which follows from (\ref{V-nat of K_2}). The coherence conditions for $\lf{K}$ to be a $\V$-functor, i.e.,
\begin{align}
\lf{K}_{V\oV W,\alg{M}}\ci\asso_{V,W,\lf{K}\alg{M}}&=\lf{K}\asso_{V,W,\alg{M}}\ci\lf{K}_{V,W\oV\alg{M}}\ci(V\oV\lf{K}_{W,\alg{M}})\\
\lf{K}_{I,\alg{M}}\ci\luni_{\lf{K}\alg{M}}&=\lf{K}\luni_{\alg{M}}
\end{align}
reduce to the analogous relations for $K$ if we apply $G'$ and use (\ref{def of strength of Kbar}) and strictness of the $\V$-functor $G$,
i.e., the identities $G\asso_{V,W,\alg{M}}=\asso_{V,W,G\alg{M}}$ and $G\luni_{\alg{M}}=\luni_{G\alg{M}}$.

It remains to show that the skew monoidal structure of $\lf{K}$ constructed in the first part consists of $\V$-natural transformations.
For $\lf{K_0}$ there is nothing to prove but for $\lf{K}_{\alg{M},\alg{N}}$ we have to prove $\V$-naturality in the first and the second argument,
\begin{align}
\lf{K}\lf{\Gamma}_{V,\alg{M},\alg{N}}\ci\lf{K}_{V,\alg{M}\hot\alg{N}}\ci(V\oV\lf{K}_{\alg{M},\alg{N}})&=
\lf{K}_{V\oV \alg{M},\alg{N}}\ci(\lf{K}_{V,\alg{M}}\hot\lf{K}\alg{N})\ci\lf{\Gamma}_{V,\lf{K}\alg{M},\lf{K}\alg{N}}\\
\lf{K}\lf{\Gamma'}_{V,\alg{M},\alg{N}}\ci\lf{K}_{V,\alg{M}\hot\alg{N}}\ci(V\oV\lf{K}_{\alg{M},\alg{N}})&=
\lf{K}_{\alg{M},V\oV\alg{N}}\ci(\lf{K}\alg{M}\hot\lf{K}_{V,\alg{N}})\ci\lf{\Gamma'}_{V,\lf{K}\alg{M},\lf{K}\alg{N}}\,.
\end{align}
Fortunately, these relations immediately reduce to the analogous relations for $K$, e.g. to relation (\ref{V-nat of K_2}), after
applying the forgetful strict $\V$-functor $G'$.
This finishes the construction of the lifting $\lf{K}$ of $K$ as a skew monoidal $\V$-functor.
\end{proof}

As an application of the above Lifting Theorem we can show that a mere category equivalence proven in \cite[Theorem 5.3]{SMC} 
is actually a skew monoidal equivalence. 
Let $E$ be the endomorphism monoid of the skew monoidal unit $R$, so $E$ is a monoid in $\V$.
In \cite{SMC} we have shown that any skew monoidal structure $\bra\M,\smp,R\ket$ on $\M$ induces a skew monoidal structure 
$\bra\,_E\M,\smpq,\Eob{R}\ket$ on the category of $E$-objects and there is a skew monoidal forgetful functor $\phi:\,_E\M\to\M$. 
Let $T^q$ denote the canonical monad of $\bra\,_E\M,\smpq,\Eob{R}\ket$.
\begin{thm} \label{thm: supplement of SMC}
For any r2-exact skew monoidal tensored $\V$-category $\bra\M,\smp,R\ket$ there is a strictly commutative diagram of skew monoidal $\V$-functors
\[
\begin{CD}
(\,_E\M)^{T_q}@>\bar\phi>\sim>\M^T\\
@V{G^q}VV @VV{G}V\\
_E\M@>\phi>>\M
\end{CD}
\]
in which $\bar\phi$ is an equivalence of skew monoidal $\V$-categories, in particular a strong skew monoidal functor.
\end{thm}
Since this result somewhat deviates from the main topic of the paper, its proof is relegated to the Appendix.

\section{The underlying category} \label{sec: Eob}

The aim of this section is the construction of the analogue of the bimodule category $\Bimod{R}{R}$ underlying the module category of a bialgebroid over $R$. Since $\M$ plays the role of the right $R$-module category $\Mod{R}$ and $\Bimod{R}{R}$ is the Eilenberg-Moore category of the monad $R\oV\under$ on $\Mod{R}$,
the natural candidate is the category $_E\M$ of left $E$-objects in $\M$ where $E:=\M(R,R)$ is the endomorphism monoid of the object $R$. However the (skew) monoidal structure
of $_E\M$ should be independent of the given $\smp$-structure on $\M$, just like the monoidal category $\Bimod{R}{R}$ exists independently of what kind of $R$-bialgebroids we
want to consider. This implies e.g. that the skew monoidal product $\smpq$ of the Appendix has nothing to do with the skew monoidal product 
$\uot$ we are looking for. A skew monoidal structure $\bra\,_E\M,\uot,\Eob{R}\ket$ on the category of $E$-objects which is determined solely by 
the $\V$-category structure of $\M$ and by the position of the object $R$ in $\M$ is called the underlying category, provided an appropriate skew
monoidal forgetful functor $\bra\M^T,\hot,\alg{R}\ket\to \bra\,_E\M,\uot,\Eob{R}\ket$ exists.

We start with the following observation.
\begin{pro}\label{pro: dot}
For every object $R$ in a tensored $\V$-category $\M$ there is a skew monoidal structure on $\M$ with unit object $R$.
Using the functor $H:=\M(R,\under):\M\to \V$ the skew monoidal product is
\[
M\bullet N\ :=\ HM\oV N
\]
and the coherence morphisms are
\begin{align}
\label{dotgamma}
\dot\gamma_{L,M,N}&=(H_{HL,M}\oV N)\ci\asso_{HL,HM,N}\\
\dot\eta_M&=(i_R\oV M)\ci\luni_M\\
\label{doteps}
\dot\eps_M&=\ev_{R,M}\,.
\end{align}
Furthermore, the functor $H$, equipped with the natural transformation
\[
H_{HM,N}:HM\oV HN\to H(M\bullet N)
\]
and the arrow $I\rarr{i_R}HR$, is a skew monoidal functor $\bra\M,\bullet,R\ket\to \bra\V,\oV,I\ket$.

If $\M$ has and $V\oV\under$ preserves reflexive coequalizers for all $V\in\V$ then $\bra\M,\bullet,R\ket$ is r2-exact.
\end{pro}
\begin{proof}
The proof is a straightforward application of the coherence conditions 
\begin{align}
H a_{V,W,M}\ci H_{V,W\oV M}\ci(V\oV H_{W,M}&=H_{V\oV W,M}\ci a_{V,W,HM}\\
H_{I,M}\ci l_{HM}&=H l_M
\end{align}
for the strength $H_{V,M}:=\chi_{V,R,M}$ introduced in (\ref{def of chi}) and in particular the relations 
\begin{align*}
H\ev_{R,M}\ci H_{HM,R}&=c_{R,R,M}\\
H\ev_{R,M}\ci H_{HM,R}\ci(HM\oV i_R)&=\runi_{HM}\,.
\end{align*}
The details are left to the reader. As far as the r2-exactness is concerned, if $V\oV\under$ preserves reflexive coequalizers
for all $V\in\V$ then clearly does $M\bullet\under$ the same for all $M\in\M$.
Thus, for r2-exactness of $\bra\M,\bullet,R\ket$ we need only that the dotted $\Gamma'$ be an isomorphism. 
Since $\bullet$ is obtained from the action, as a 2-variable strong morphism $\under\oV\under:\V\x\M\to\M$, by composing
with $H$ on the left leg, strongness on the right leg remains untouched. Indeed, the strengths of $\bullet$ are
\begin{alignat}{2}
\dot\Gamma_{V,M,N} &=(H_{V,M}\oV N)\ci\asso_{V,HM,N} 
&\quad : V\oV(M\bullet N) &\to (V\oV M)\bullet N\\
\dot\Gamma'_{V,M,N}&=\asso^{-1}_{HM,V,N}\ci(s_{V,HM}\oV N)\ci\asso_{V,HM,N} 
&\quad: V\oV(M\bullet N) &\to M\bullet (V\oV N)\,.
\end{alignat}
\end{proof}
Now we are ready to apply Theorem \ref{thm: hot} to $\bra\M,\bullet,R\ket$ and to its canonical monad $\dot T=R\bullet\under=E\oV\under$.
\begin{cor} \label{cor: uot}
Let $\M$ be a tensored $\V$-category with reflexive coequalizers and assume that $V\oV\under$ preserves reflexive coequalizers for all $V\in\V$.
Then the category $\M^{\dot T}$ of modules over the skew monoidal category $\bra\M,\bullet,R\ket$ is the category $_E\M$
of $E$-objects equipped with `horizontal' tensor product $\Eob{M}\uot\Eob{N}$ defined by the coequalizer
\begin{equation} \label{uot as hot}
\M(R,M)\oV (E\oV N)\longerpair{(c_{R,R,M}\oV N)\ci\asso_{HM,E,N}}{HM\,\oV\, \lambda }
\M(R,M)\oV N\coequalizer{\omega_{M,\Eob{N}}}M\uot\Eob{N}
\end{equation}
and with skew unit object $\Eob{R}=\bra R,\ev_{R,R}\ket$.
\end{cor}
\begin{proof}
The canonical monad of the dot-structure is $\dot T=E\oV\under$ and it has multiplication 
\begin{align*}
\dot\mu_{L,M}&\equiv (\dot\eps_L\bullet M)\ci\dot\gamma_{L,R,M}=\\
&=(H\ev_{R,L}\oV M)\ci (H_{HL,R}\oV M)\ci\asso_{HL,E,M}=\\
&\eqby{c via ev and H}(c_{R,R,L}\oV M)\ci\asso_{HL,E,M}
\end{align*}
which shows that (\ref{uot as hot}) is identical to the coequalizer (\ref{def of hot'}) defining the horizontal tensor product
for $E$-objects.

The skew monoidal structure of $_E\M$ obtained by applying Theorem \ref{thm: hot} to $\bra\M,\bullet,R\ket$ has
skew monoidal unit $\Eob{R}=\bra R, \ev_{R,R}\ket$ and its coherence morphisms $\Eob{\gamma}$, $\Eob{\eta}$ and $\Eob{\eps}$ are the unique arrows making the diagrams 
\[
\begin{CD}
L\bullet (M\bullet N)@>\omega_{\Eob{L},\Eob{M}\uot\Eob{N}}\ci(L\bullet \omega_{\Eob{M},\Eob{N}})>>
\Eob{L}\uot(\Eob{M}\uot\Eob{N})\\
@V{\dot\gamma_{L,M,N}}VV @VV{\Eob\gamma_{\Eob{L},\Eob{M},\Eob{N}}}V\\
(L\bullet M)\bullet N@>\omega_{\Eob{L}\uot\Eob{M},\Eob{N}}\ci(\omega_{\Eob{L},\Eob{N}}\bullet N)>>
(\Eob{L}\uot\Eob{M})\uot\Eob{N}
\end{CD}
\]
\vskip 0.5cm
\begin{equation} \label{diag: Eob eta eps}
\begin{CD}
M@=M\\
@V{\dot\eta_M}VV @VV{\Eob{\eta}_{\Eob{M}}}V\\
R\bullet M@>\omega_{\Eob{R},\Eob{M}}>> \Eob{R}\uot\Eob{M}
\end{CD}
\qquad\qquad
\begin{CD}
M\bullet R@>\omega_{\Eob{M},\Eob{R}}>> \Eob{M}\uot\Eob{R}\\
@V{\dot\eps_M}VV @VV{\Eob{\eps}_{\Eob{M}}}V\\
M@=M
\end{CD}
\end{equation}
commutative. 
\end{proof}

The skew monoidal structure of $_E\M$ as presented by the above Corollary does not allow to say too much about whether
the coherence morphisms are invertible, except for $\Eob{\eta}$, which is always invertible by Theorem \ref{thm: hot} (i). 
In the next Proposition we give an alternative description for which, however, we need some preparations.

Let $\bra\, _E\V_E,\oE,E\ket$ be the monoidal category of $E$-$E$-bimodules in $\V$. 
The coherence isomorphisms of $_E\V_E$ are denoted by $\Eob{\asso}$, $\Eob{\luni}$ and $\Eob{\runi}$, all oriented according to our skew monoidal convention.

The category of right $E$-modules $\V_E$ is a right $_E\V_E$-category in the usual way, with the action defined by a choice of
coequalizers
\[
V\oV(E\oV W)\pair{}{}V\oV W\coequalizer{}\Eob{V}\oE\Eob{W}
\]
of reflexive pairs in $\V$ for each $\Eob{V}\in\V_E$ and $\Eob{W}\in\,_E\V_E$. The tensor product $\Eob{U}\oE\Eob{W}$ of two $E$-$E$-bimodules can be chosen by lifting this action to $_E\V_E$ along the forgetful functor $\phi:\,_E\V_E\to\V_E$ in such a way that $\phi$ to become a strict morphism of right $_E\V_E$-actegories. 

By a similar choice of coequalizers in $\M$ we can define a functor $\V_E\x\,_E\M\to\M$, also denoted by $\oE$,
\[
V\oV(E\oV M)\longpair{(\rho\oV M)\ci\asso_{V,E,M}}{V\oV\lambda}
V\oV M\coequalizer{\tau_{\Eob{V},\Eob{M}}}\Eob{V}\oE\Eob{M}
\]
for each $\Eob{V}\in\V_E$ and $\Eob{M}\in\,_E\M$. By $\V$-naturality of the reflexive pair we can lift the result to $_E\M$
if $\Eob{V}$ is an $E$-$E$-bimodule. In this way $\oE$ becomes an action making $_E\M$ to a left $_E\V_E$-actegory.
The coherence isomorphisms are denoted by $\Eob{\asso}_{\Eob{U},\Eob{V},\Eob{M}}$ and $\Eob{\luni}_{\Eob{M}}$ and
satisfy the usual actegory axioms.
For $\Eob{V}\in\,_E\V_E$ and $\Eob{M}\in\,_E\M$ the coequalizers $\tau_{\Eob{V},\Eob{M}}$ become components of a strength
\begin{equation}
\dot G_{\Eob{V},\Eob{M}}:=\tau_{\Eob{V},\Eob{M}}\ :\ \phi\Eob{V}\oV \dot G\Eob{M}\ \to\ \dot G(\Eob{V}\oE\Eob{M})
\end{equation}
for $\dot G:\,_E\M\to\M$ relative to $\phi:\,_E\V_E\to\V$. That is to say, $_\phi\dot G$ is a morphism from the $_E\V_E$-actegory $_E\M$ to the $\V$-actegory $\M$ in the sense of Definition \ref{def: gen act mor}.
The corresponding coherence condition (\ref{gen-strength-1}) takes the form
\begin{equation}\label{phi Gdot}
\dot G_{\Eob{V}\oE\Eob{W},\Eob{M}}\ci(\phi_{\Eob{V},\Eob{W}}\oV \dot G\Eob{M})\ci
\asso_{\phi\Eob{V},\phi\Eob{W},\dot G\Eob{M}}=
\dot G\Eob{\asso}_{\Eob{V},\Eob{W},\Eob{M}}\ci\dot G_{\Eob{V},\Eob{W}\oE\Eob{M}}\ci
(\phi\Eob{V}\oV\dot G_{\Eob{W},\Eob{M}})
\end{equation}
which will be needed soon.
Disregarding from the fact that $_E\V_E$ is not symmetric, $_E\M$ is in fact a tensored $_E\V_E$-category with $\M(\Eob{M},\Eob{N})\ :\ E\oV E^\op\to\V$ playing the role of the $_E\V_E$-valued hom.

Let $J$ denote the $\V$-functor
\[
J:=\M(\Eob{R},\under):\M\to\V_E
\]
whose value on the object $M$ is the right $E$-module $HM$ with action $c_{R,R,M}:HM\oV E\to HM$.
Its strength is the lifting of $H_{V,M}$ to $\V_E$,
\begin{align*}
&J_{V,M}=\\
&\left( V\oV JM\equiv\bra V\oV HM,(V\oV c_{R,R,M})\ci\asso^{-1}_{V,HM,E}\ket\rarr{H_{V,M}}
\bra H(V\oV M),c_{R,R,V\oV M}\ket\equiv J(V\oV M)\right)\,.
\end{align*}
$J$ has a left adjoint $J^*:=\under\oE\Eob{R}:\V_E\to\M$. The notation $\Eob{\ev}$ and $\Eob{\coev}$ will be used for the counit and unit of $J^*\dashv J$.

By definition, the object $R$ is called a \textsl{dense generator} for the $\V$-category $\M$ if $J$ is fully faithful \cite[Vol. 1]{Borceux}. 

We introduce also the functor $\JJ:\,_E\M\to\,_E\V_E$ which associates to $\Eob{M}$ the composite $\V$-functor $E\rarr{\Eob{M}}\M\rarr{J}\V_E$.
So, it is legitimate to write $\JJ\Eob{M}=J\Eob{M}$. The purpose of using a new notation is that we want to make $\JJ$ a $_E\V_E$-functor.
This can be done by requiring strict commutativity of the square
\[
\begin{CD}
_{_E\V_E}(\,_E\M)@>_\phi{\dot G}>>_\V\M\\
@V{\JJ}VV @VV{H}V\\
_{_E\V_E}(\,_E\V_E)@>_\phi\phi>>_\V\V
\end{CD}
\]
in the 2-category $\ACT$. This means, beyond the equality $H\dot G=\phi \JJ$ of functors, that the strength of $\JJ$ is determined by 
\begin{equation} \label{J-H-square}
\begin{CD}
\phi\Eob{V}\oV \overset{\phi \JJ}{\overbrace{H\dot G}}\Eob{M}@>\phi_{\Eob{V},\JJ\Eob{M}}>>\phi(\Eob{V}\oE \JJ\Eob{M})\\
@V{H_{\phi\Eob{V},\dot G\Eob{M}}}VV @VV{\phi \JJ_{\Eob{V},\Eob{M}}}V\\
H(\phi\Eob{V}\oV\dot G\Eob{M})@>H\dot G_{\Eob{V},\Eob{M}}>>\underset{H\dot G}{\underbrace{\phi \JJ}}(\Eob{V}\oE\Eob{M})
\end{CD}
\end{equation}

We want to tie the definition of $\uot$ to the definition of $\oE$ by setting the coequalizer $\omega$ of (\ref{uot as hot})
to be 
\begin{equation} \label{omega via tau}
\omega_{L,\Eob{M}}\ :=\ \tau_{JL,\Eob{M}}\qquad L\in\M,\ \Eob{M}\in\,_E\M\,.
\end{equation}
This choice allows the following pleasant representation of the coherence morphisms of $\uot$.
\begin{pro} \label{pro: uot via J}
The skew monoidal structure $\bra\,_E\M,\uot,\Eob{R},\Eob{\gamma},\Eob{\eta},\Eob{\eps}\ket$ on the category of $E$-objects can be chosen in such a way that the skew monoidal product is 
\[
\Eob{L}\uot\Eob{M}\ :=\JJ\Eob{L}\oE \Eob{M}
\]
and the coherence morphisms are
\begin{align}
\label{formula ugamma}
\Eob{\gamma}_{\Eob{L},\Eob{M},\Eob{N}}&=(\JJ_{J\Eob{L},\Eob{M}}\oE\Eob{N})\ci\Eob{\asso}_{J\Eob{L},J\Eob{M},\Eob{N}}\\
\label{formula ueta}
\Eob{\eta}_{\Eob{M}}&=\Eob{\luni}_{\Eob{M}}\\
\label{formula ueps}
\Eob{\eps}_{\Eob{M}}&=\Eob{\ev}_{\Eob{R},\Eob{M}}\,.
\end{align}
These expressions imply the following invertibility properties:
\begin{enumerate}
\item $\Eob{\eta}$ is an isomorphism.
\item $\Eob{\eps}$ is an isomorphism if $R$ is a dense generator in $\M$ as a $\V$-category.
\item $\Eob{\gamma}$ is an isomorphism if and only if $\JJ_{J\Eob{L},\Eob{M}}\oE\Eob{N}$ is an isomorphism for all
$E$-objects $\Eob{L}$, $\Eob{M}$ and $\Eob{N}$.  
\end{enumerate}
\end{pro}
\begin{proof}
The choice (\ref{omega via tau}) clearly entails the equality $L\uot\Eob{M}\ :=\ JL\oE \Eob{M}$.

If $\Eob{M}=\bra M,\lambda\ket$ then by the equality $\omega_{R,\Eob{M}}=\tau_{E,\Eob{M}}$ the unique arrow $u$ such that $u\ci\omega_{R,\Eob{M}}=\lambda$ is at the same time $\Eob{\eta}^{-1}_{\Eob{M}}$, just as in the proof of Theorem \ref{thm: hot} (i), and 
$\Eob{\luni}^{-1}_{\Eob{M}}$, by definition of the coherence isomorphism $\Eob{\luni}$ of the $_E\V_E$-actegory $_E\M$. This proves
(\ref{formula ueta}) from which (i) is obvious.

The counit $\Eob{\ev}_{\Eob{R},\Eob{M}}$ of $\JJ^*\dashv \JJ$ is related to the counit of $\under\oV R\dashv H$ by
\[
\dot G\Eob{\ev}_{\Eob{R},\Eob{M}}\ci\tau_{JM,\Eob{R}}\ =\ \ev_{R,M}\,.
\]
By (\ref{omega via tau}) and by the 2nd diagram of (\ref{diag: Eob eta eps}) the same equation defines $\Eob{\eps}_{\Eob{M}}$. This proves (\ref{formula ueps}) from which (ii) immediately follows.

In order to prove (\ref{formula ugamma}) we recall the proof of Corollary \ref{cor: uot} that $\Eob{\gamma}$ is defined by
the equation
\begin{equation} \label{def of ugamma}
\dot G\Eob{\gamma}_{\Eob{L},\Eob{M},\Eob{N}}\ci\omega_{\Eob{L},\Eob{M}\uot\Eob{N}}\ci
(\dot G\Eob{L}\bullet \omega_{\Eob{M},\Eob{N}})
=\omega_{\Eob{L}\uot\Eob{M},\Eob{N}}\ci(\omega_{\Eob{L},\Eob{N}}\bullet \dot G\Eob{N})
\ci\dot\gamma_{\dot G\Eob{L},\dot G\Eob{M},\dot G\Eob{N}}\,.
\end{equation}
It suffices to show that the expression (\ref{formula ugamma}) satisfies this equation.
Substituting $\omega_{\Eob{M},\Eob{N}}=\tau_{J\Eob{M},\Eob{N}}=\dot G_{J\Eob{M},\Eob{N}}$ and the expression (\ref{dotgamma}) into the LHS of (\ref{def of ugamma}) the claim follows from the following computation:
\begin{align*}
&\dot G(\JJ_{J\Eob{L},\Eob{M}}\oE\Eob{N}) \ci \dot G\Eob{\asso}_{J\Eob{L},J\Eob{M},\Eob{N}}\ci 
\dot G_{J\Eob{L},J\Eob{M}\oE \Eob{N}}\ci(\phi J\Eob{L}\oV\dot G_{J\Eob{M},\Eob{N}})\\
&\eqby{phi Gdot}\dot G(\JJ_{J\Eob{L},\Eob{M}}\oE\Eob{N}) \ci \dot G_{J\Eob{L}\oE J\Eob{M},\Eob{N}}\ci
(\phi_{J\Eob{L},J\Eob{M}}\oV \dot G\Eob{N})\ci
\asso_{\phi J\Eob{L},\phi J\Eob{M},\dot G\Eob{N}} \\
&=\dot G_{J(J\Eob{L}\oE \Eob{M}),\Eob{N}}\ci(\phi \JJ_{J\Eob{L},\Eob{M}}\oV\dot G\Eob{N})\ci 
(\phi_{J\Eob{L},J\Eob{M}}\oV \dot G\Eob{N})\ci\asso_{H\dot G\Eob{L}, H\dot G\Eob{M},\dot G\Eob{N}} \\
&\eqby{J-H-square}\dot G_{J(J\Eob{L}\oE \Eob{M}),\Eob{N}}\ci (H\dot G_{J\Eob{L},\Eob{M}}\oV\dot G\Eob{N})
\ci(H_{H\dot G\Eob{L},\dot G\Eob{M}}\oV \dot G\Eob{N})\ci\asso_{H\dot G\Eob{L}, H\dot G\Eob{M},\dot G\Eob{N}}
\end{align*}
which is the RHS of (\ref{def of ugamma}) by (\ref{dotgamma}). This finishes the proof of (\ref{formula ugamma}) from which (iii) is obvious.
\end{proof}

\begin{rmk}
The coherence conditions for the strength $\JJ_{\Eob{V},\Eob{M}}:\Eob{V}\oE J\Eob{M}\to J(\Eob{V}\oE\Eob{M})$ are
\begin{align}
\JJ_{\Eob{V}\oE\Eob{W},\Eob{M}}\ci\Eob{\asso}_{\Eob{V},\Eob{W},J\Eob{M}}
&=J\Eob{\asso}_{\Eob{V},\Eob{W},\Eob{M}}\ci \JJ_{\Eob{V},\Eob{W}\oE\Eob{M}}\ci(\Eob{V}\oE \JJ_{\Eob{W},\Eob{M}})\\
\JJ_{\Eob{E},\Eob{M}}\ci\Eob{\luni}_{J\Eob{M}}&=J\Eob{\luni}_{\Eob{M}}
\end{align}
where $\Eob{V},\Eob{W}\in\,_E\V_E$ and $\Eob{M}\in\,_E\M$. The second relation immediately implies that $\JJ_{\Eob{E},\Eob{M}}$ is an isomorphism 
and so is $\Eob{\gamma}_{\Eob{R},\Eob{M},\Eob{N}}$ for all $\Eob{M},\Eob{N}\in\,_E\M$. This is in complete agreement with invertibility of $\Eob{\eta}$
since the latter implies invertibility of $\Eob{\gamma}_{\Eob{R},\Eob{M},\Eob{N}}$ by the skew monoidal triangle axiom (\ref{SMC2}).

Similarly, if we assume invertibility of $\Eob{\ev}_{\Eob{R},\under}$ the $_E\V_E$-naturality relation (\ref{V-nat of ev}) implies
that $\JJ^*\JJ_{\Eob{V},\Eob{M}}$ is invertible for all $\Eob{V}\in\,_E\V_E$ and $\Eob{M}\in\,_E\M$ and therefore
$\Eob{\gamma}_{\Eob{L},\Eob{M},\Eob{R}}$ is an isomorphism for all $\Eob{L},\Eob{M}\in\,_E\M$.
This is again something that follows directly from invertibility of $\Eob{\eps}$ and the skew monoidal triangle (\ref{SMC3}).

Still assuming invertibility of $\Eob{\ev}_{\Eob{R},\under}$ the $\Eob{\coev}_{\Eob{R},J\Eob{M}}$ is an isomorphism by the adjunction relation. Hence the $_E\V_E$-naturality condition (\ref{V-nat of coev}) implies that 
\begin{equation}\label{monoidal sheaves}
\Eob{\coev}_{\Eob{R},J\Eob{L}\oE J\Eob{M}}\ \text{isomorphism}\quad\Leftrightarrow\quad
\JJ_{J\Eob{L},\Eob{M}}\ \text{isomorphism}
\end{equation}
for any pair of objects $\Eob{L},\Eob{M}\in\,_E\M$.

The latter equivalence together with the observation that the skew monoidal structure of Proposition \ref{pro: uot via J} is a special case of the Altenkirch-Chapman-Uustalu construction
\cite{Altenkirch-Chapman-Uustalu} will be the clue to characterize in Section \ref{monoidality} the situation of the underlying category being monoidal and 
$\JJ:\,_E\M\to \,_E\V_E$ being a strong monoidal functor.
\end{rmk}

\section{The forgetful functor}

This section is about the forgetful functor that, for a (right) $R$-bialgebroid $B$, maps a right $B$-module to the underlying $R$-$R$-bimodule. This functor is strong monoidal and one of the main questions of this paper is whether similar forgetful functors exist on module categories $\M^T$ of r2-exact skew monoidal $\V$-categories $\M$.

The target of the forgetful functor has been constructed in the previous section. Therefore the functor $\Forg$ we are looking for is of the type $\M^T\to\,_E\M$. Since both the target and source are Eilenberg-Moore categories of canonical monads, it seems natural to define $\Forg$ by lifting 
a very elementary functor $\M\to\M$ which merely connects $\smp$ with $\bullet$. 
\begin{lem} \label{lem: sigma}
Let $\bra\M,\smp,R\ket$ be a skew monoidal tensored $\V$-category and let $\bra\M,\bullet,R\ket$ be the skew monoidal $\V$-category structure associated to the tensored $\V$-category $\M$ and to the object $R$ by Proposition \ref{pro: dot}.
Then the identity functor $\M\to\M$ together with the $\V$-natural transformation
\[
\sigma_{M,N}:=\left(
\begin{CD}
M\bullet N@>HM\oV\eta_N>>HM\oV TN@>\Gamma_{HM,R,N}>>(HM\oV R)\smp N@>\ev_{R,M}\smp N>>M\smp N
\end{CD}
\right)
\]
and the identity arrow $R\to R$ is a skew monoidal $\V$-functor $\bra\M,\smp,R\ket\longrarr{\Sigma}\bra\M,\bullet,R\ket$.
\end{lem}
\begin{proof}
The skew monoidal functor axioms
\begin{align}
\gamma_{L,M,N}\ci\sigma_{L,M\smp N}\ci(L\bullet\sigma_{M,N})
&=(\sigma_{L,M}\smp N)\ci\sigma_{L\bullet M,N}\ci\dot\gamma_{L,M,N}\\
\sigma_{R,M}\ci\dot\eta_M&=\eta_M\\
\eps_M\ci\sigma_{M,R}&=\dot\eps_M
\end{align}
can be verified by a straightforward calculation.
\end{proof}
By \cite[Lemma 2.7]{SMC} the above Lemma implies that
\[
\sigma_M:=\sigma_{R,M}\ :\ R\bullet M\to R\smp M
\]
is a monad morphism from $E\oV\under$ to $T$.
Therefore we define the forgetful functor $\Forg$ as the functor induced by $\sigma_M$ on the Eilenberg-Moore categories,
\begin{equation} \label{def of forg}
\Forg:\M^T\to\,_E\M,\quad \bra M,\alpha\ket\mapsto\bra M, \alpha\ci\sigma_M\ket\,.
\end{equation}
In other words, $\Forg$ is the lifting, in the sense of Theorem \ref{thm: lift}, of the functor defined in Lemma \ref{lem: sigma}.

\begin{thm} \label{thm: G}
Let $\M$ be an r2-exact skew monoidal tensored $\V$-category. Then the forgetful functor $\Forg$ defined in
(\ref {def of forg}) is a skew monoidal $\V$-functor making the diagram
\[
\begin{CD}
\M^T@>\Forg>>_E\M\\
@V{G}VV @VV{\dot G}V\\
\M@>\Sigma>>\M
\end{CD}
\]
strictly commutative as skew monoidal $\V$-functors. 
Furthermore, $\Forg$ is monadic as a $\V$-functor giving rise to a $\V$-monad (neither skew monoidal nor skew opmonoidal in general)  $\Eob{T}$ on $_E\M$ such that 
$\M^T\cong (\,_E\M)^{\Eob{T}}$.
\end{thm}
\begin{proof}
$\Forg$ is the lift in the sense of Theorem \ref{thm: lift} of the functor $\Sigma=\bra \M,\sigma,R\ket$ constructed in Lemma \ref{lem: sigma}. 
Therefore it is a skew monoidal $\V$-functor such that $\dot G\,\Forg=\Sigma G$. Merely as functors, however, we have $\dot G\,\Forg=G$ with both $G$ and $\dot G$ monadic.
Since all categories have reflexive coequalizers, it follows by standard arguments relying on Beck's Theorem that $\Forg$ is monadic precisely if it has a left adjoint. 

As for the existence of the left adjoint $\F$ we can rely on classical adjoint lifting theorems such as \cite[Theorem 3.7.3]{TTT}. It is interesting that
a left adjoint can be explicitely given in terms of the quotient skew monoidal product (\ref{q coeq}) as 
\begin{equation} \label{coeq rho lambda2}
\F:\,_E\M\to\M^T,\quad \F\Eob{M}=\bra R\smpq\Eob{M},\nabla_{\F\Eob{M}}\ket
\end{equation}
where the $T$-algebra structure is given by unique factorization in the diagram
\begin{equation}
\label{def of nabla F}
\parbox{300pt}{
\begin{picture}(300,120)
\put(-20,95){$E\oV(R\smp(R\smp M))$}
\put(70,96){\vector(1,0){60}} \put(95,84){$\lambda_3$}
\put(70,100){\vector(1,0){60}} \put(95,106){$\rho_2$}
\put(140,95){$ R\smp(R\smp M)$}
\put(200,98){\vector(1,0){60}} \put(215,105){$Tq_{R,\Eob{M}}$}
\put(270,95){$T(R\smpq\Eob{M})$}

\put(20,80){\vector(0,-1){50}} \put(25,54){$E\oV\mu_{R,M}$}
\put(160,80){\vector(0,-1){50}} \put(165,54){$\mu_{R,M}$}
\dashline{3}(290,80)(290,32)\put(290,32){\vector(0,-1){0}} \put(295,54){$\nabla_{\F\Eob{M}}$} 

\put(-5,10){$ E\oV (R\smp M)$}
\put(70,11){\vector(1,0){60}} \put(95,0){$\lambda_2$}
\put(70,15){\vector(1,0){60}} \put(95,21){$\rho$}
\put(147,10){$ R\smp M$}
\put(200,13){\vector(1,0){60}} \put(215,20){$ q_{R,\Eob{M}}$}
\put(275,10){$ R\smpq\Eob{M}$}
\end{picture}}
\end{equation}
where the notation $\rho_i$, $\lambda_i$ is that of Appendix \ref{sec: smpq}.
Notice that the first row is not exactly the $T$ of the second, it is modified by a harmless composition with $\Gamma'$ (which is an isomorphism due to one of the r2-exactness conditions). 

In order to make $\F$ a $\V$-functor it suffices to supply it with a strength $\F_{V,\Eob{M}}:V\oV\F\Eob{M}\to\F(V\oV\Eob{M})$. 
This can be done in the more general case of diagram (\ref{Gamma'^q}), so $\F_{V,\Eob{M}}=\Gamma^{\prime q}_{V,R,\Eob{M}}$.
This strength is in fact invertible, because $\Gamma'$ is, in complete agreement with the fact that a left adjoint actegory morphism must always be strong. The ordinary adjunction $\F\dashv\G$ follows from \cite[Theorem 3.7.3]{TTT} with the remark that the existence of reflexive coequalizers
suffices.
\end{proof}

In the remaining part of this section we study the question of when $\Forg$ is strong skew monoidal. Since $\Forg_0:\Eob{R}\to\Forg\alg{R}$ is the identity by (\ref{lift K0}) and by $\Sigma_0$ being the identity, the question reduces to studying $\Forg_{\alg{M},\alg{N}}$ the definition of which, according to (\ref{lift K2}), can be read from the diagram
\begin{equation} \label{diag: G2}
\parbox{300pt}{
\begin{picture}(300,120)
\put(0,95){$M\bullet(R\bullet N)$}
\put(70,96){\vector(1,0){60}} \put(80,84){$M\bullet\lambda_{\Forg\alg{N}}$}
\put(70,100){\vector(1,0){60}} \put(85,106){$\dot\mu_{M,N}$}
\put(143,95){$ M\bullet N$}
\put(190,98){\vector(1,0){60}} \put(205,105){$\omega_{M,\Forg\alg{N}}$}
\put(260,95){$M\uot\Forg\alg{N}$}

\put(10,80){\vector(0,-1){50}} \put(15,54){$\sigma_{M,TN}\ci(M\bullet\sigma_N)$}
\put(160,80){\vector(0,-1){50}} \put(165,54){$\sigma_{M,N}$}
\dashline{3}(280,80)(280,32)\put(280,32){\vector(0,-1){0}} \put(285,54){$\Forg_{M,\alg{N}}$} 

\put(-5,10){$M\smp(R\smp N)$}
\put(70,11){\vector(1,0){60}} \put(80,0){$M\smp\nabla_{\alg{N}}$}
\put(70,15){\vector(1,0){60}} \put(85,21){$\mu_{M,N}$}
\put(147,10){$ M\smp N$}
\put(190,13){\vector(1,0){60}} \put(205,20){$\pi_{M,\alg{N}}$}
\put(265,10){$ M\hot\alg{N}$}
\end{picture}}
\end{equation}
by lifting  $\Forg_{M,\alg{N}}\in\M$ to $\Forg_{\alg{M},\alg{N}}\in\,_E\M$.
The unlifted $\Forg_{M,\alg{N}}$-s are the components of a strength for the identity functor $\id_\M$ that forms, together with the skew monoidal $\G$, an actegory morphism from the right $\M^T$-actegory $\M$ to the right $_E\M$-actegory $\M$ (a right actegory version of Definition 
\ref{def: gen act mor}).
Let us introduce the notation
\[
\lambda_{M,N}:=(\ev_{R,M}\smp N)\ci\Gamma_{HM,R,N}\ :\ M\bullet (R\smp N)\to M\smp N
\]
so that $\sigma_{M,N}=\lambda_{M,N}\ci (M\bullet \eta_N)$. The special case $\lambda_{R,N}:E\oV TN\to TN$ returns the $E$-object obtained from the free $T$-algebra 
$FN=\bra TN,\mu_N\ket$ by applying $\G$. Indeed, $\mu_N\ci\sigma_{TN}=\lambda_{R,N}$. In fact more is true:
\begin{align}
\notag
\mu_{M,N}\ci\sigma_{M,TN}
&=(\eps_M\smp N)\ci\gamma_{M,R,N}\ci(\ev_{R,M}\smp TN)\ci\Gamma_{HM,R,TN}\ci(HM\oV\eta_{TN})=\\
\notag
&=(\eps_M\smp N)\ci((\ev_{R,M}\smp R)\smp N)\ci\gamma_{HM\oV R,R,N}\ci\Gamma_{HM,R,R\smp N}\ci (HM\oV\eta_{TN})=\\
\notag
&\eqby{sma-11}(\ev_{R,M}\smp N)\ci(\eps_{HM\oV R}\smp N)\ci(\Gamma_{HM,R,R}\smp N)\ci\Gamma_{HM,R\smp R,N}\ci\\
\notag
&\ci (HM\oV\gamma_{R,R,N})\ci (HM\oV\eta_{TN})=\\
\notag
&\stackrel{(\ref{SMC2})+(\ref{sma-15})} {=}
(\ev_{R,M}\smp N)\ci((HM\oV\eps_R)\smp N)\ci\Gamma_{HM,R\smp R,N}\ci(HM\oV(\eta_R\smp N))=\\
\notag
&=(\ev_{R,M}\smp N)\ci\Gamma_{HM,R,N}=\\
\label{lambda-mu-sigma}
&=\lambda_{M,N}
\end{align}
for all objects $M$, $N$ of $\M$.

\begin{pro} \label{pro: G strong = lambda coeq}
For an r2-exact skew monoidal tensored $\V$-category $\M$ the
strength $\Forg_{M,\alg{N}}:M\uot\Forg\alg{N}\to M\hot\alg{N}$ is invertible
if and only if $\lambda_{M,N}$ is a coequalizer of the reflexive pair
\[
M\bullet(R\bullet TN)\longpair{\dot\mu_{M,TN}}{M\bullet\lambda_{R,N}}M\bullet TN
\]
for all $M,N\in\M$.
\end{pro}
\begin{proof}
Inserting a free $T$-algebra $FN=\bra TN,\mu_N\ket$ for $\alg{N}$ in the definition (\ref{diag: G2}) of $\Forg_{M,\alg{N}}$ and composing it with 
the isomorphism (\ref{the j}) we obtain by (\ref{lambda-mu-sigma}) that
\[
j_{M,N}\ci\Forg_{M,FN}\ci\omega_{M,\Forg FN}=j_{M,N}\ci\pi_{M,FN}\ci\sigma_{M,TN}=\mu_{M,N}\ci\sigma_{M,TN}=\lambda_{M,N}\,.
\]
Therefore $\Forg_{M,FN}$ is invertible precisely if $\lambda_{M,N}$ is a coequalizer of the above pair. 
In this way we reduced the problem to prove invertibility of $\Forg_{M,\alg{N}}$ for all $\alg{N}\in\M^T$ whenever invertibility of $\Forg_{M,FN}$ is known for all $N\in\M$. 
Writing any $T$-algebra $\alg{N}$ as a reflexive coequalizer of free $T$-algebras and using naturality of $\Forg_{M,\alg{N}}$ we have a commutative diagram
\begin{equation}
\parbox{300pt}{
\begin{picture}(300,120)
\put(5,95){$M\uot\Forg FTN$}
\put(70,96){\vector(1,0){60}} \put(85,84){$\sst M\uot \Forg T\nabla_{\alg{N}}$}
\put(70,100){\vector(1,0){60}} \put(85,106){$\sst M\uot\Forg\mu_{N}$}
\put(138,95){$ M\uot \Forg FN$}
\put(190,98){\vector(1,0){60}} \put(205,105){$\sst M\uot\Forg\nabla_{\alg{N}}$}
\put(260,95){$M\uot\Forg\alg{N}$}

\put(20,80){\vector(0,-1){50}} \put(25,54){$\sst\Forg_{M,FTN}$}
\put(160,80){\vector(0,-1){50}} \put(165,54){$\sst\Forg_{M,FN}$}
\put(275,80){\vector(0,-1){50}} \put(280,54){$\sst\Forg_{M,\alg{N}}$} 

\curvedashes[1mm]{0,1,2}
\put(290,25){\curve(0,0, 30,20, 30,40, 0,60)}
\curvedashes{} 
\put(290,85){\vector(-2,1){0}}

\put(5,10){$M\hot FTN$}
\put(70,11){\vector(1,0){60}} \put(85,0){$\sst M\hot T\nabla_{\alg{N}}$}
\put(70,15){\vector(1,0){60}} \put(85,21){$\sst M\hot\mu_{N}$}
\put(142,10){$ M\hot FN$}
\put(190,13){\vector(1,0){60}} \put(205,20){$\sst M\hot\nabla_{\alg{N}}$}
\put(265,10){$ M\hot\alg{N}$}
\end{picture}}
\end{equation}
The second row is a coequalizer since $\hot$ is r2-exact and this leads to a unique (dashed) arrow that makes the diagram, with $\Forg_{M,\alg{N}}$ removed, commutative.
The first row is the result of applying $M\uot\Forg\under$ to a reflexive coequalizer which is in fact $G$-contractible. Using the relation $\dot G\Forg=G$ we see that $\Forg$ maps
$G$-contractible pairs to $\dot G$-contractible ones, so by monadicity of $\dot G$ the arrow $\Forg\nabla_{\alg{N}}$ is the coequalizer of a reflexive pair in $_E\M$.
By r2-exactness of $\uot$ we conclude that the first row is also a coequalizer, hence the dashed arrow is necessarily the inverse of $\Forg_{M,\alg{N}}$.
\end{proof}

\begin{cor} \label{cor: G strong}
Let $\bra\M,\smp,R\ket$ be an r-exact skew monoidal $\V$-category in which $R$ is a dense generator for $\M$. 
Then $\Forg:\M^T\to\,_E\M$ is a strong skew monoidal functor.
\end{cor}
\begin{proof}
The statement follows by considering the commutative diagram
\begin{alignat}{2}
&HL\oV(E\oV TM)\ \ \longerpair{(c_{R,R,L}\oV TM)\ci\asso_{HL,E,TM}}{HL\oV\lambda_{R,M}}
&HL\oV TM \qquad\longrarr{\lambda_{L,M}}\quad
& L\smp M \notag\\ 
\label{lambda as coeq}
&\qquad\quad\parbox[c]{5pt}{\begin{picture}(5,40)\put(2,40){\vector(0,-1){40}}\end{picture}}\sst \Gamma\ci(1\oV \Gamma)
&\parbox[c]{5pt}{\begin{picture}(5,40)\put(2,40){\vector(0,-1){40}}\end{picture}}
\sst \Gamma\qquad\qquad\qquad\qquad\quad
&\quad\parbox[c]{5pt}{\begin{picture}(5,40)\put(0,40){\line(0,-1){40}}\put(3,40){\line(0,-1){40}}\end{picture}}\\
&(HL\oV(E\oV R))\smp M\, \longerpair{[(c_{R,R,L}\oV R)\ci\asso]\smp M}{(HL\oV\ev_{R,R})\smp M}
&(HL\oV R)\smp M\quad \longrarr{\ev_{R,L}\smp M}\quad
&L\smp M \notag
\end{alignat}
in which the 2nd row is obtained by applying $\under\smp M$ to a reflexive coequalizer. The vertical arrows are all isomorphisms
by the assumption that $\smp$ is strong in the 1st argument. Thus the first row is a coequalizer, too, and Proposition \ref{pro: G strong = lambda coeq} applies.
\end{proof}

The following Theorem just summarizes the content of this Section.
\begin{thm} \label{thm: G strong}
Let $\bra\M,\smp,R\ket$ be an r-exact skew monoidal tensored $\V$-category and assume that $R$ is a dense generator for $\M$. 
Then the forgetful functor $\Forg:\M^T\to\,_E\M$ is strong skew monoidal and monadic and $\Eob{T}=\Forg\F$ is a skew opmonoidal $\V$-monad on $_E\M$.
\end{thm}

\section{Self-cocompleteness and monoidality} \label{monoidality}

Assuming the forgetful functor $\Forg:\M^T\to\,_E\M$ is strong skew monoidal the question of whether $\hot$ on $\M^T$ is a monoidal product 
can be answered affirmatively provided $\uot$ on $_E\M$ is monoidal. The latter being independent of the original skew monoidal product $\smp$ on $\M$ we see that it is only the object $R$ in the bare $\V$-category $\M$ which we should study.

We consider a somewhat more general problem by replacing $E$ with a small $\V$-category $\C$. In this way we will be able to
comment also about the Altenkirch-Chapman-Uustalu skew monoidal structure on the functor category $[\C,\M]$. 
Also we relax, or just forget about, the existence of (arbitrary) tensors in $\M$. But we will need certain weighted colimits to exist in $\M$. In the absence of tensors the colimit $U\star F$ of $F:\C\to\M$ weighted by $U:\C^\op\to\V$ cannot be represented by a coend $\int^CUC\oV FC$. Instead it is defined \cite{Kelly} as the representing object in
\[
\M(U\star F,M)\cong[\C^\op,\V](U,\M(F\under,M)),\quad M\in\M.
\]

\begin{defi}
Let $\C$ be a small $\V$-category. A full replete subcategory $\W\subseteq[\C^\op,\V]$ is called self-cocomplete if
\begin{description}
\item[(sc-1)] $\W$ contains all the representable functors $\C(\under,C)$, $C\in\C$ and
\item[(sc-2)] $\W$ is closed under $\W$-weighted colimits (of $\V$-functors $\C\to\W$).
\end{description}
\end{defi}
\begin{rmk} \label{rmk: sc}
Since $\W$ is full, condition (sc-1) is equivalent to that the Yoneda embedding $C\mapsto YC=\C(\under,C)$ factors through the inclusion $\W\subseteq[\C^\op,\V]$ via a unique (fully faithful)
$Y_\W:\C\to\W$.
Condition (sc-2), in turn, says that $U\star F$, which always exists in $[\C^\op,\V]$ and can be represented there as a pointwise coend $D\mapsto\int^CUC\oV FCD$, belongs to $\W$ 
for all $U\in\W$ and $F\in[\C,\W]$. 
Using the fact that fully faithful functors reflect colimits the weighted colimit $U\star F$ is not only a colimit in $[\C^\op,\V]$ but also in $\W$, provided $U\in\W$, $F\in[\C,\W]$.
Thus (sc-2) can also be read as that $\W$ has $\W$-weighted colimits and the inclusion $\W\subseteq[\C^\op,\V]$ preserves them. Even if colimits, other than $\W$-weighted ones, 
happen to exist in $\W$ these need not be preserved by the inclusion.
\end{rmk}
\begin{lem} \label{lem: Eob(W)}
For $\W\subseteq[\C^\op,\V]$ a self-cocomplete full replete subcategory and for a $\V$-functor $F:\C\to[\C^\op,\V]$ the following conditions are equivalent:
\begin{enumerate}
\item $U\star F\in\W$ for all $U\in\W$.
\item $FC\in\W$ for all $C\in\C$.
\end{enumerate}
\end{lem}
\begin{proof}
Since $\W$ is full, (ii) implies that $F$ factors through the inclusion $\W\subseteq[\C^\op,\V]$ therefore (i) follows from assumption (sc-2) on $\W$.
Vice versa, if (i) holds then setting $U$ to be representable $FC\cong YC\star F\in\W$ for all $C$.
\end{proof}

\begin{cor} \label{cor: Eob(W)}
For $\W\subseteq[\C^\op,\V]$ a full, replete and self-cocomplete subcategory define $\Eob{\W}:=[\C,\W]$. Then $\Eob{\W}$ is a monoidal subcategory of the bimodule category
$[\C\oV\C^\op,\V]$ via the embedding
\[
[\C,\W] \subseteq [\C,[\C^\op,\V]]\cong [\C\oV\C^\op,\V]
\]
and $\W$ is a right $\Eob{\W}$-actegory via the weighted colimit operation $\star$.
\end{cor}
\begin{proof}
$\Eob{\W}$ consists precisely of the functors $F:C\mapsto \{D\mapsto F(C,D)\in\V\}$ satisfying condition (ii) of Lemma \ref{lem: Eob(W)}. By condition (i) of the same Lemma 
it is clear that $\Eob{\W}$ is monoidal with monoidal product
\[
(F\ot G )(B,D)=\int^CF(B,C)\oV G(C,D)
\]
and with monoidal unit $Y_\W(C,D)=\C(D,C)$.
\end{proof}

\begin{exa}
For $\V$-functors $U:\C^\op\to\V$ and $X:\D\to\V$ with $\C$ and $\D$ small the colimits $U\star\under$ weighted by $U$ exist both in $\V$ and $[\D,\V]$ and the limits 
$\{X,\under\}$ weighted by $X$ exist both in $\V$ and $[\C^\op,\V]$. If for all $\V$-functors $G:\D\oV\C\to\V$ there is an isomorphism
\[
U\star\{X,G\}\ \cong\ \{X,U\star G\}
\]
then one says that $U$-weighted colimits commute with $X$-weighted limits in $\V$. This is equivalent to saying that the functor $U\star\under:[\D,\V]\to\V$ preserves $X$-weighted
limits and also to that the functor $\{X,\under\}:[\C^\op,\V]\to\V$ preserves the $U$-weighted colimits.

For a fixed $X$ let $\W_X\subseteq[\C^\op,\V]$ be the full subcategory of weights $U$ such that $U\star\under$ preserves $X$-weighted limits.  
Then it is evident that $\W_X$ is a self-cocomplete replete full subcategory of $[\C^\op,\V]$.
\end{exa}

\begin{exa} \label{exa: flat}
The category $\Flat(\C^\op,\V)$ consists of functors $F$ such that the $F$-weighted colimits commute with all finite limits.
Just like in the previous example it is easy to see that $\Flat(\C^\op,\V)\subseteq[\C^\op,\V]$ is a self-cocomplete replete full subcategory.
Similar can be said about the category $\alpha$-$\Flat(\C^\op,\V)$ of $\alpha$-flat functors (a.k.a. $\alpha$-filtered weights) where $\alpha$ is a regular cardinal. For the precise
definitions of these notions we refer to \cite{Borceux-Quinteiro-Rosicky} which reveals also that in order for these examples to work we have to assume that $\V$ is a locally $\alpha_0$-presentable base and $\alpha\geq \alpha_0$. Many of the usual base categories are in fact locally finitely presentable bases.
\end{exa}

\begin{rmk}
$\V$-categories equivalent to self-cocomplete subcategories of presheaves have a certain resemblance to accessible categories \cite{Adamek-Rosicky,Borceux-Quinteiro-Rosicky}
after a little reformulation. 
Let $\C\into\M$ be a full small subcategory. Let $\W'\subseteq[\C^\op,\V]$ be any full replete subcategory such that
\begin{itemize}
\item $\M$ has $\W'$-weighted colimits,
\item each $C\in \C$ is "$\W'$-presentable", i.e., $\M(C,\under):\M\to\V$ preserves $\W'$-weighted colimits,
\item every $M\in\M$ is the $\W'$-weighted colimit of the inclusion functor $\C\into\M$.
\end{itemize}
Then there is a full replete self-cocomplete $\W\subseteq[\C^\op,\V]$ (in fact $\W\subseteq\W'$) such that $\M\simeq\W$ as extensions of $\C$.  
\end{rmk}

In the next two examples the small category $\C$ is a one-object $\V$-category, i.e., a monoid $E$ in $\V$.
\begin{exa} \label{exa: fgp}
Assume $\V$ is preadditive, i.e., an $\Ab$-category. 
Let $\W_{\fgp}\subseteq\V_E$ be the full subcategory of finitely generated projective right $E$-modules. Then the right regular $E_E$ belongs to $\W_\fgp$. Let $\phi:\,_E\V_E\to\V_E$ 
denote  the forgetful functor. If $U\in\W_\fgp$ and $V\in\,_E\V_E$ is such that $\phi V\in\W_\fgp$ then there exist dual bases, i.e., finite diagrams
\[
U\rarr{f_i}E_E\rarr{u_i}U\quad\text{such that}\quad\sum_{i=1}^nu_i\ci f_i=U
\]
and
\[
\phi V\rarr{g_j}E_E\rarr{v_j}\phi V\quad\text{such that}\quad\sum_{j=1}^mv_j\ci g_j=\phi V\,.
\]
Therefore
\[
U\oE V\longrarr{f_i\oE V}E\oE V\iso\phi V\rarr{g_j}E_E\rarr{v_j}\phi V\iso E\oE V\longrarr{u_i\oE V}U\oE V
\]
is also a dual basis proving that $\W_\fgp$ is self-cocomplete.
\end{exa}

\begin{exa} \label{exa: rank 1}
Let $\W_1\subseteq\V_E$ be the full subcategory of rank 1 free right $E$-modules. If $U\in\W_1$ and $V\in\,_E\V_E$ is such that $\phi V\in\W_1$ then
$U\oE V\cong E\oE V\cong\phi V\cong E_E$. Therefore $\W_1$ is self-cocomplete. 
\end{exa}

For $\V=\Ab$ and $E$ a ring, many constructions of ring theory lead also to self-cocomplete subcategories of $\V_E$ such as finitely generated modules,
projective modules and finitely presented modules. 

\begin{exa}
This is an example of a $\W\subseteq\Ab_E$ which is not self-cocomplete. Let $E$ be a field and let $\W$ be the full subcategory of $E$-vector spaces of dimension less than $n$ where $n>2$. Let $E^2$ be the 2-dimensional vector space and define the functor $F:E\to\V_E$, i.e., the bimodule $F$, as $E^2$ with left $E$-action equal to the right $E$-action.
Then the $W$-weighted colimit $W\star F$ is just the tensor product $W\ot E^2$ of vector spaces and has dimension $2\dim W$. Hence $\W$ satisfies (sc-1) but not (sc-2).
\end{exa}

Now let $\M$ be an arbitrary $\V$-category and $\C$ be a small $\V$-category. For a $\V$-functor $J:\C\to\M$ we define $J_*:\M\to[\C^\op,\V]$ as the $\V$-functor sending $M$ to the
presheaf $\M(J\under,M)$. Let $\W\subseteq[\C^\op,\V]$ denote the full replete subcategory of objects $J_*M$, $M\in\M$. 
Assume that
\begin{description}
\item[(wb-0)] the $\W$-weighted colimits of any functor $F:\C\to\M$ exist in $\M$.
\end{description}
Then we can introduce (the enriched version of) the skew monoidal structure that Altenkirch, Chapman and Uustalu construct \cite{Altenkirch-Chapman-Uustalu} 
on the functor category $[\C,\M]$. (For the precise connection we note that the left Kan extension of $G$ along $J$ becomes $\Lan_JG=G^*J'_*$ in our notation.)

The skew monoidal product of $F$ and $G\in[\C,\M]$ is a (chosen) $\W$-weighted pointwise colimit $(F\ot G)C := J_*(FC)\star G$, i.e., it is the composite functor
\[
F\ot G\ :=\ \left(\C\rarr{F}
\M\rarr{J'_*}\W\rarr{G^*}\M\right)
\]
where we introduced $G^*:=\under\star G$, the weighted colimit of $G$ as a functor of the weight, and $J'_*$ denotes $J_*$ with codomain restricted to $\W$.
Notice that then $J^*$ is left adjoint to $J'_*$. Still, when no confusion arises, we write $J_*$ in place of $J'_*$.

The monoidal unit is the object $J\in[\C,\M]$.

The skew monoidal comparison morphisms are 
\begin{align*}
\eps_F&:F\ot J\to F, \quad (\eps_F)_C=\ev_{FC}:J^*J_*FC\to FC\\
\eta_G&:G\to J\ot G,\quad (\eta_G)_C=\left(GC\rarr{\sim}G^*YC\rarr{G^*J}G^*J_*JC\right)\\
\gamma_{F,G,H}&:F\ot(G\ot H)\to(F\ot G)\ot H,\\
&=\left((H^*J_*G)^*J_*F\rarr{\sim}H^*(J_*G)^*J_*F\longrarr{H^*J_{J_*F,G}}H^*J_*G^*J_*F\right)
\end{align*}
where $\ev_M:\M(J\under,M)\star J\to M$ is the counit of the adjunction $J^*\dashv J'_*$, the isomorphism in $\gamma_{F,G,H}$ comes from interchangeability of colimits and 
$J_{U,G}:U\star J_*G\to J_*(U\star G)$ is the $\V$-natural transformation
\begin{align*}
I&\longrarr{i_{U\star G}}\M(U\star G,U\star G)\iso[\C^\op,\V](U,\M(G\under,U\star G))\to\\
&\longrarr{(J_*)\under}[\C^\op,\V](U,[\C^\op,\V](J_*G\under,J_*(U\star G)))\iso[\C^\op,\V](U\star J_*G,J_*(U\star G))
\end{align*}
Invertibility of $J_{U,G}$ means precisely that $J_*$ preserves the $U$-weighted colimit of $G$.

The authors of  \cite{Altenkirch-Chapman-Uustalu} call $J:\C\to\M$ \textsl{well-behaved} if, beyond the condition (wb-0) above, the following three conditions hold.
\begin{description}
\item[(wb-1)] $J:\C\to\M$ is fully faithful.
\item[(wb-2)] $J_*:\M\to[\C^\op,\V]$ is fully faithful.
\item[(wb-3)] $J_*$ preserves the $\W$-weighted colimits of all $G:\C\to \M$.
\end{description}
It is then proved in \cite[Theorem 4.4]{Altenkirch-Chapman-Uustalu} that if $J$ is well-behaved then the skew monoidal category $\bra[\C,\M],J,\ot,\gamma,\eta,\eps\ket$ defined above is a monoidal category. In the above presentation this is rather clear: (wb-1) implies that $\eta$ is invertible, (wb-2) implies that $\eps$ is invertible and (wb-3) implies that $\gamma$ is invertible.

\begin{thm} \label{thm: wb-sc}
Let $\M$ be a $\V$-category and let $\C$ be a small $\V$-category. Then there exists a well-behaved $J:\C\to\M$ precisely when $\M$ is equivalent to a self-cocomplete full replete subcategory $\W\subseteq[\C^\op,\V]$.
In this case the monoidal category $\bra[\C,\M],J,\ot,\gamma,\eta,\eps\ket$ is equivalent to the monoidal category $\Eob{\W}$ of those bimodules $F:\C\oV\C^\op\to\V$ which
as right $\C$-modules belong to $\W$.
\end{thm}
\begin{proof}
If a well-behaved $J:\C\to\M$ exists then the fully faithful $J_*$, with codomain restricted to its replete image $\W$, defines an equivalence $\M\simeq\W$. By (wb-1) $Y\cong J_*J$ therefore
$\W$ satisfies axiom (sc-1) of self-cocompleteness. Since $J_*$ preserves the $\W$-weighted colimits by (wb-3), all of which exists in $\M$ by (wb-0), the $\W$ satisfies (sc-2). Therefore
$\W$ is self-cocomplete.

Assume that $\Psi:\M\iso\W$ is an equivalence with a self-cocomplete full replete $\W\subseteq[\C^\op,\V]$. Then $\Psi$ is part of an adjoint equivalence $\Psi^*\dashv\Psi$.
Define $J:=\Psi^*Y_\W:\C\to\M$ which is fully faithful obviously, so (wb-1) holds. Then the associated $J_*=\M(J,\under)$ satisfies
\begin{align*}
J_*M&=\M(\Psi^*Y_\W\under,M)\cong\W(Y_\W\under,\Psi M)=\\
&=[\C^\op,\V](Y_,\Psi M)\cong\Psi M
\end{align*}
by Yoneda Lemma therefore $J_*$ has replete image precisely $\W$. Moreover the corestriction $J'_*:\M\to\W$ is isomorphic to $\Psi$ implying that $J_*$ is fully faithful, so (wb-2) holds 
for $J$. The colimit of $F:\C\to\M$ weighted by $U\in\W$ can now be constructed from its existence in $\W$,
\[
U\star F:=\Psi^*(U\star \Psi F)\,.
\]
Indeed,
\begin{align*}
\M(U\star F,M)&\cong\W(U\star\Psi F,\Psi M)\cong [\C^\op,\V](U,\W((\Psi F\under,\Psi M))\cong\\
&\cong[\C^\op,\V](U,\M(F\under,M))\,.
\end{align*}
This proves (wb-0). It remains to show (wb-3). First notice that the inclusion $\W\subseteq[\C^\op,\V]$ both preserves and reflects $\W$-weighted colimits.
Therefore $J_*$ preserves the $\W$-weighted colimits iff $J'_*$ does. But $J'_*$ is isomorphic to the equivalence $\Psi:\M\iso\W$ so it preserves every colimit that exist in $\M$.
This proves (wb-3) hence also the first statement. As for the second, let $J:\C\to\M$ be well-behaved and let $\W$ be the self-cocomplete replete image of $J_*$.
Then composing functors with $J_*$ defines a functor 
\[
[\C,\M]\to [\C,\W]\equiv\Eob{\W},\quad F\mapsto J'_*F
\]
which is an equivalence since $J'_*$ is. But it is also a strong monoidal functor because $J'_*J\cong Y_\W$ and 
\begin{equation*}
J_*F\ot J_*G=\int^C\M(JC,F\under)\oV\M(J\under,GC)\cong\M(J\under,\int^C\M(JC,F\under)\oV GC)=J_*(F\ot G)
\end{equation*}
due to property (wb-3).
\end{proof}

The above Theorem is worth a comparison with the general theory of free cocompletions \cite{Kelly,Albert-Kelly} which offers an elegant explanation for the monoidal structure on $[\C,\M]$. Let $\Phi=\{\phi_i:\J_i\to\V\,|\,i\in\I\}$ be a family of small weights. Then for a small category
$\C$ there is the closure $\Phi(\C)\subseteq[\C^\op,\V]$ of $\C$ under $\Phi$-colimits. Let $\Phi$-$\mathsf{Cocts}$ denote the 2-category of 
$\Phi$-cocomplete $\V$-categories, $\Phi$-cocontinuous functors and all natural transformations between them. Then composition with the
inclusion functor $J:\C\to\Phi(\C)$ defines an equivalence \cite[Prop. 4.1.]{Albert-Kelly}
\[
\Phi\text{-}\mathsf{Cocts}(\Phi(\C),\N)\ \iso\ [\C,\N]
\]
for any $\Phi$-cocomplete $\N$. Inserting $\N=\Phi(\C)$ and using the fact that the endohom category of any 2-category is a monoidal category
the above equivalence induces a monoidal structure on $[\C,\Phi(\C)]$. By \cite[Prop. 4.1.]{Albert-Kelly} or \cite[Theorem 5.35]{Kelly} we know
that the inverse equivalence is given by left Kan extension along $J$. Therefore it is not difficult to see that the induced monoidal structure of 
$[\C,\Phi(\C)]$ is precisely the skew monoidal structure of \cite{Altenkirch-Chapman-Uustalu} with a well-behaved $J$.

Since $\Lan_JF$ of a functor $F\in[\C,\Phi(\C)]$ is given pointwise by a colimit 
\[
\Lan_JF(M)\ =\ J_*M\star F,\qquad M\in\Phi(\C)
\]
the weight $J_*M$ of which belongs to $\Phi(\C)$, the very existence of $\Lan_JF$ shows that $\Phi(\C)$ is self-cocomplete. Vice versa, if $\W\subseteq[\C^\op,\V]$ is 
self-cocomplete then, obviously, $\W$ is the $\Phi$-cocompletion of $\C$ with $\Phi=\W$. Thus being self-cocomplete is a property
that detects whether a full replete subcategory of $[\C^\op,\V]$ is the $\Phi$-cocompletion of $\C$ w.r.t. some family $\Phi$ of weights.

Applying the above notions and results to our problem of monoidality of the module category of a skew monoidal tensored $\V$-category we 
meet only very special self-cocomplete subcategories $\W$ of $[E^\op,\V]=\V_E$, namely reflective ones. But even for a cocomplete $\W$
the condition of self-cocompleteness is non-trivial since the inclusion $\W\into\V_E$ need not preserve all (small) colimits.

\begin{pro} \label{pro: main}
Let $\M$ be a tensored $\V$-category with reflexive coequalizers and with a dense generator $R$.
Let $E=\M(R,R)$ and let the category $\V_E$ be endowed with the skew monoidal product $V\oV' W:=\phi'V\oV W$, where $\phi':\,_E\V_E\to\V_E$
is the forgetful functor.
\begin{enumerate}
\item $\M$ is a reflective subcategory of the presheaf category $\V_E$, hence complete and cocomplete.
\item The inclusion functor $J:\M\into\V_E$, $M\mapsto \M(\Eob{R},M)$, exhibits $\bra\M,\bullet,R\ket$ as the skew monoidal reflection of
$\bra\V_E,\oV',E'\ket$ in the sense of \cite[Theorem 2.1]{Lack-Street: lift}.
\item The lift of $J:\M\into\V_E$ to the Eilenberg-Moore categories is a functor $\JJ:\,_E\M\to\,_E\V_E$ which exhibits $\bra\,_E\M,\uot,\Eob{R}\ket$
as the skew monoidal reflection of the ordinary monoidal $\V$-category $\bra\,_E\V_E,\oE,E\ket$.
\item If $\JJ:\,_E\M\to\,_E\V_E$ is strong monoidal then $\bra\,_E\M,\uot,\Eob{R}\ket$ is a closed monoidal reflective subcategory of the bimodule category $\bra\,_E\V_E,\oE,E\ket$ and $\JJ$ is strong closed.
\item If the replete image of $J:\M\into\V_E$ is a self-cocomplete subcategory of $\V_E$ then the functor $\JJ:\,_E\M\to\,_E\V_E$ is strong monoidal.
\end{enumerate}
\end{pro}
\begin{proof}
(i) Existence of tensors and reflexive coequalizers in $\M$ imply the existence of the colimits $V\oE\Eob{R}$ for all $V\in\V_E$ and this defines
a left adjoint $J^*$ of $J$.

(ii) The condition in \cite[Theorem 2.1]{Lack-Street: lift} for $J^*\dashv J$ to be a (right) skew monoidal reflection is invertibility of the arrows
$J^*(JM\oV'\coev_{\Eob{R},V})$ for $M\in\M$, $V\in\V_E$, which holds automatically due to invertibility of $\coev_{\Eob{R},V}\oE\Eob{R}$.
The skew monoidal product induced by the reflection is isomorphic to $\bullet$ since 
\[
M\bullet N=HM\oV N\cong HM\oV(JN\oE\Eob{R})\cong J^*(\phi'JM\oV JN)=J^*(JM\oV' JN)\,.
\]

(iii) It is easy to verify that the Eilenberg-Moore category of the canonical monad for $\bra\V_E,\oV',E'\ket$ is precisely $\bra\,_E\V_E,\oE,E\ket$.
The lift of $J:\M\into\V_E$ is the functor $\JJ:\,_E\M\into\,_E\V_E$ of composition with $J$  
and it has a left adjoint $\JJ^*\Eob{V}=\Eob{V}\oE\Eob{R}$. The condition for $\bra\,_E\M,\uot,\Eob{R}\ket$ to arise from $\bra\,_E\V_E,\oE,E\ket$
as a skew reflection is invertibility of $\JJ^*(\JJ\Eob{M}\oE\Eob{\coev}_{\Eob{R},\Eob{V}})$ which holds automatically as in (ii). Then the 
isomorphism $\Eob{M}\uot\Eob{N}\cong \JJ^*(\JJ\Eob{M}\oE\JJ\Eob{N})$ can be verified easily.

(iv) Conservativity of the fully faithful $\JJ$ implies that $\bra\,_E\M,\uot,\Eob{R}\ket$ is a monoidal category. 
Since $\bra\,_E\V_E,\oE,E\ket$ is closed, we are in the situation of a closed reflection \cite{Day} and the adjunction $\JJ^*\dashv\JJ$ induces 
left and right internal homs on $\bra\,_E\M,\uot,\Eob{R}\ket$ and $\JJ$ preserves them.

(v) The skew monoidal structure of $\JJ$ defined by the reflection consists of the identity $E=\JJ\Eob{R}$ and of the arrows 
$\JJ_{M,N}:=\Eob{\coev}_{\Eob{R},{J}\Eob{M}\oE {J}\Eob{N}}$. By assumption the $J$ preserves the $JM$-weighted colimits and this implies in particular that the strengths $\JJ_{J\Eob{M},\Eob{N}}$ are invertible. Using (\ref{monoidal sheaves}) this means that $\JJ$ is strong monoidal.
\end{proof}

In the Theorem below we say that an object $R\in\M$ \textsl{well-generates} $\M$ if the corresponding $\V$-functor $E\to\M$ is well-behaved.
\begin{thm} \label{thm: main}
Let $\bra\M,\smp,R\ket$ be an r-exact skew monoidal tensored $\V$-category  in which $\M$ is
well-generated by $R$. Let $E$ be the endomorphism monoid of $R$. Then 
\begin{enumerate}
\item the $\V$-category $\M$ is complete and cocomplete and is equivalent to a self-cocomplete full replete and reflective subcategory
$\W\subseteq\V_E$,
\item $\bra\,_E\M,\uot,\Eob{R}\ket$ is a closed monoidal $\V$-category monoidally equivalent to the subcategory $\Eob{\W}\subseteq\,_E\V_E$ of those bimodules which, as right $E$-modules, belong to $\W$.
\item the skew monoidal category $\bra\M^T,\hot,\alg{R}\ket$ of modules over $\M$ is a monoidal category
and the forgetful functor $\G:\M^T\to\,_E\M$ is strong monoidal and monadic,
\item the monoidal category $\bra\M^T,\hot,\alg{R}\ket$ is equivalent to the Eilenberg-Moore category 
of an opmonoidal monad defined on $\Eob{\W}$.
\end{enumerate}
\end{thm}
\begin{proof}
(i) follows from Proposition \ref{pro: main} (i) and Theorem \ref{thm: wb-sc}. 

(ii) By Proposition \ref{pro: main} (v) and (iv) $\bra\,_E\M,\uot,\Eob{R}\ket$ is closed monoidal and using (i) $_E\M=[E,\M]$ is equivalent to the
monoidal category $[E,\W]=\Eob{\W}$ by Corollary \ref{cor: Eob(W)}.

(iii) Using that well-generators are dense generators we obtain from Theorem \ref{thm: G strong} that $\Forg$ is strong skew monoidal and monadic. 
Since the codomain of $\G$ is monoidal and $\G$ reflects isomorphisms, the skew monoidal functor relations imply that the 
domain of $\G$ is also a monoidal category.

(iv) The equivalence of $\M^T$ with $\Eob{\W}^{\Eob{T}}$ for an opmonoidal monad follows also from Theorem \ref{thm: G strong} because $_E\M$ is equivalent to $\Eob{\W}$.
\end{proof}

\appendix

\section{Enriched $E$-objects}  \label{sec: smpq}

In \cite{SMC} we have shown for $\V=\Set$ that any skew monoidal structure $\bra\M,\smp,R\ket$ on $\M$ induces a skew monoidal structure 
$\bra\,_E\M,\smpq,\Eob{R}\ket$ on the category of $E$-objects.
This Appendix serves as a supplement to \cite[Theorem 5.3]{SMC} in which the equivalence of $\M^T$ with $(_E\M)^{T_q}$ is proven merely as a category equivalence. 
Here we show, using the Lifting Theorem, that this is actually a skew monoidal equivalence and this holds for general $\V$. 
Also we need to reformulate some Lemmas of \cite{SMC} in order to adapt them to the enriched setting.

Let $\bra\bra\M,c,i\ket, \bra \oV,a,l\ket, \bra\smp,\Gamma,\Gamma'\ket, R,\gamma,\eta,\eps\ket$ be a skew monoidal tensored $\V$-category.

Let $E:=\bra\M(R,R),c_{R,R,R},i_R\ket$ be the endomorphism monoid of the skew unit object $R$. The algebras for the monad $E\oV\under$ will be called $E$-objects. So, the
$E$-objects are pairs $\bra M,\lambda\ket$ where $M$ is an object of $\M$ and $\lambda:E\oV M\to M$ is an arrow satisfying
\begin{align}
\label{E-ob-1}
\lambda\ci(E\oV\lambda)&=\lambda\ci(c_{R,R,R}\oV M)\ci a_{E,E,M}\\
\label{E-ob-2}
\lambda\ci(i_R\oV M)\ci l_M&=M\,.
\end{align}
A morphism $\bra M,\lambda\ket\to\bra M',\lambda'\ket$ of $E$-objects is an arrow $t:M\to M'$ satisfying $\lambda'\ci(E\oV t)=t\ci\lambda$. 

The category of $E$-objects $_E\M$ can alternatively be defined as the category $[E,\M]$ of $\V$-functors from the 1-object $\V$-category $E$ to $\M$. The equivalence of the two definitions
is provided by the adjunction relations 
\begin{align}
\varphi&=\M(M,\lambda)\ci\coev_{M,E}\\
\lambda&=\ev_{M,M}\ci(\varphi\oV M)
\end{align}
between the $\V$-functor data $\varphi:E\to\M(M,M)$ and the action data $\lambda:E\oV M\to M$.

Using the symmetry $s$ in $\V$ we can define the monoid $E^\op:=\bra\M(R,R),c^\op_{R,R,R},i_R\ket$, with opposite multiplication $c^\op_{R,R,R}=c_{R,R,R}\ci s_{E,E}$,
and $E^\op$-objects as pairs $\bra M,\rho\ket$ where $\rho:E\oV M\to M$ is the same type of arrow as in an $E$-object but satisfies different axioms, namely (\ref{E-ob-1}) and
(\ref{E-ob-2}) with $c$ replaced by $c^\op$. In other words, $E^\op$-objects are $\V$-functors $E^\op\to\M$.

$U$-objects in $\M$ for an arbitrary monoid $\bra U,u_2,u_0\ket$ in $\V$ can be defined similarly. 
If $U$ and $V$ are monoids in $\V$ then $W=U\oV V$ is also a monoid with multiplication $w_2$ which is, up to invertible associators, the composite $(u_2\oV v_2)\ci(U\oV s_{V,U}\oV V)$.
Given a $U$-object $\bra M,\kappa\ket$ and a $V$-object $\bra M,\lambda\ket$ on the same underlying $M$ the necessary and sufficient condition for the existence of a $U\oV V$-object  $(U\oV V)\oV M\to M$ extending both $\kappa$ and $\lambda$ is that "$\kappa$ and $\lambda$ commute" in the sense of obeying 
\[
\kappa\ci(U\oV \lambda)=\lambda\ci(V\oV\kappa)\ci a^{-1}_{V,U,M}\ci(s_{U,V}\oV M)\ci a_{U,V,M}\,.
\]
The following Lemma is the enriched version of \cite[Lemma 4.1]{SMC} in which we use the definitions
\begin{align*}
\mu_{M,N}&:=(\eps_M\smp N)\ci\gamma_{M,R,N}\\
\sigma_N&:=(\ev_{R,R}\smp N)\ci\Gamma_{E,R,N}\ci(E\oV\eta_N)
\end{align*}
and remark that $\sigma: E\oV\under\to R\smp\under$ is a monad morphism \cite{Street: Mnd},
\begin{align}
\label{sigma 1}
\mu_N\ci\sigma_{R\smp N}\ci(E\oV\sigma_N)&=\sigma_N\ci(c_{R,R,R}\oV N)\ci a_{E,E,N}\\
\label{sigma 2}
\sigma_N\ci(i_R\oV N)\ci l_N &=\eta_N\,,
\end{align}
and $\mu$ satisfies
\begin{align}
\label{mu 1}
\mu_{M,N}\ci\mu_{M,R\smp N}&=\mu_{M,N}\ci(M\smp\mu_N)\\
\label{mu_2}
\mu_{M,N}\ci(M\smp\eta_N)&=M\smp N\,.
\end{align}
\begin{lem}
Let $M$ and $N$ be objects in a skew monoidal tensored $\V$-category $\M$. 
\begin{enumerate}
\item Then $\bra M\smp N,\rho\ket$ is an $E^\op$-object where
\[
\rho:=\mu_{M,N}\ci(M\smp\sigma_N)\ci \Gamma'_{E,M,N}\ :\ E\oV(M\smp N)\to M\smp N\,.
\]
\item If $\bra M,\lambda_M\ket$ is an $E$-object then so is $\bra M\smp N, \lambda_1\ket$, where
\[
\lambda_1:=(\lambda_M\smp N)\ci \Gamma_{E,M,N}\ :\ E\oV(M\smp N)\to M\smp N\,,
\]
and $\lambda_1$ and $\rho$ commute.
\item If $\bra N,\lambda_N\ket$ is an $E$-object then so is $\bra M\smp N, \lambda_2\ket$, where
\[
\lambda_2:=(M\smp \lambda_N)\ci \Gamma'_{E,M,N}\ :\ E\oV(M\smp N)\to M\smp N\,,
\]
and $\lambda_2$ and $\rho$ commute.
\item If $\bra M,\lambda_M\ket$ and $\bra N,\lambda_N\ket$ are $E$-objects then the above defined $\lambda_1$ and $\lambda_2$ also commute.
\end{enumerate}
\end{lem}

If we have three objects then there are two $\smp$ signs and there are two $E^\op$-actions on $L\smp(M\smp N)$
\begin{align}
\rho_1&:=\rho_{L,M\smp N}\\
\rho_2&:=(L\smp\rho_{M,N})\ci\Gamma'_{E,L,M\smp N}
\end{align}
and two $E^\op$-actions on $(L\smp M)\smp N$
\begin{align}
\rho'_1&:=(\rho_{L,M}\smp N)\ci\Gamma_{E,L\smp M, N}\\
\rho'_2&:=\rho_{L\smp M,N}\,.
\end{align}
If the three objects are $E$-objects then there are three $E$-actions on $L\smp(M\smp N)$
\begin{align}
\lambda_1&:=(\lambda_L\smp(M\smp N))\ci\Gamma_{E,L,M\smp N}\\
\lambda_2&:=(L\smp(\lambda_M\smp N))\ci(L\smp\Gamma_{E,M,N})\ci\Gamma'_{E,L,M\smp N}\\
\lambda_3&:=(L\smp(M\smp\lambda_N))\ci(L\smp\Gamma'_{E,M,N})\ci\Gamma'_{E,L,M\smp N}
\end{align}
and three $E$-actions on $(L\smp M)\smp N$
\begin{align}
\lambda'_1&:=((\lambda_L\smp M)\smp N)\ci(\Gamma_{E,L,M}\smp N)\ci\Gamma_{E,L\smp M,N}\\
\lambda'_2&:=((L\smp\lambda_M)\smp N)\ci(\Gamma'_{E,L,M}\smp N)\ci\Gamma_{E,L\smp M,N}\\
\lambda'_3&:=((L\smp M)\smp\lambda_N)\ci \Gamma'_{E,L\smp M,N}\,.
\end{align}

\begin{lem}
For $E$-objects $\Eob{L}$, $\Eob{M}$ and $\Eob{N}$ the skew monoidal structure maps obey the following identities. 
\begin{align}
\gamma_{L,M,N}\ci\rho_i&=\rho'_i\ci(E\oV\gamma_{L,M,N})\quad(i=1,2)\\
\label{1}
\gamma_{L,M,N}\ci\lambda_i&=\lambda'_i\ci(E\oV\gamma_{L,M,N})\quad(i=1,2,3)\\
\label{eta lambda1 rho}
\rho_{R,M}\ci(E\oV\eta_M)&=\lambda_1\ci(E\oV\eta_M)\\
\label{2}
\lambda_2\ci(E\oV\eta_M)&=\eta_M\ci\lambda_M\\
\eps_M\ci\rho_{M,R}&=\eps\ci\lambda_2\\
\label{3}
\eps_M\ci\lambda_1&=\lambda_M\ci(E\oV\eps_M)
\end{align}
where the unit object $R$ is considered as an $E$-object via $\lambda_R=\ev_{R,R}:E\oV R\to R$.
\end{lem}
\begin{proof}
Relations (\ref{1}), (\ref{2}), (\ref{3}) are immediate consequences of the skew monoidality axioms (\ref{sma-11})-(\ref{sma-15}) that express the fact that $\gamma$, $\eta$, $\eps$ are transformations of actegory morphisms.
The remaining relations are also not difficult. 
\end{proof}

Now assume that $\M$ is r2-exact and define the skew monoidal product $\smpq$ of two $E$-objects by the reflexive coequalizer
\begin{equation}\label{q coeq}
E\oV(M\smp N)\pair{\rho}{\lambda_2}M\smp N\coequalizer{q_{M,\Eob{N}}} M\smpq\Eob{N}
\end{equation}
in $\M$ which can then be lifted to the $E$-object $\Eob{M}\smpq\Eob{N}=\bra M\smpq\Eob{N},\lambda_1\ket$.
The unit object for $\smpq$ is the $E$-object $\Eob{R}=\bra R,\ev_{R,R}\ket$.
\begin{pro}
For any r2-exact skew monoidal tensored $\V$-category $\bra\M,\smp,R\ket$ the category $_E\M$ of $E$-objects in $\M$ has a unique skew monoidal tensored $\V$-category 
structure with skew monoidal product $\smpq$ and with skew monoidal unit $\Eob{R}$ such that the forgetful functor $\phi:\,_E\M\to\M$ together with the natural transformation $q_{M,\Eob{N}}$ of (\ref{q coeq}) and with the identity arrow $R\to\phi\Eob{R}$ is a skew monoidal functor. The resulting skew monoidal tensored $\V$-category $\bra\,_E\M,\smpq,\Eob{R}\ket$ is r2-exact.
\end{pro}
\begin{proof}
On the level of ordinary categories and functors the unique skew monoidal structure \newline
$\bra\,_E\M,\smpq,\Eob{R},\gamma^q,\eta^q,\eps^q\ket$ that makes $\phi$ skew monoidal has already been constructed
in \cite[Proposition 4.3]{SMC}. Although the right exactness conditions of \cite[Proposition 4.3]{SMC} were somewhat too generous, by inspecting the proof one 
can see that what is really needed is only the existence of reflexive coequalizers and preservation of them by all the endofunctors $M\smp\under$. 
So, it suffices here to deal with the $\V$-actegory structure.

Like for $\M^T$ in the proof of Lemma \ref{lem: M^T is V-cat}, the action of $\V$ on $_E\M$ is defined by
\[
V\oV\Eob{M}:=\bra V\oV M, (V\oV\lambda_M)\ci E^{-1}_{V,M}\ket
\]
and the $\V$-valued hom by the equalizer
\[
_E\M(\Eob{M},\Eob{N})\equalizer{}\M(M,N)\longerpair{\M(E\oV M,\lambda_N)\ci(E\oV\under)_{M,N}}{\M(\lambda_M,N)}\M(E\oV M,N)
\]
where the $\V$-functor structure map $(E\oV\under)_{M,N}:\M(M,N)\to\M(E\oV M,E\oV N)$ is constructed  using 
(\ref{arrow map from strength}) from the strength
\begin{equation}
E_{V,M}:=a^{-1}_{E,V,M}\ci(s_{V,E}\oV M)\ci a_{E,V,M}\ :\ V\oV(E\oV M)\to E\oV(V\oV M)
\end{equation}
of the (strong) morphism $E\oV\under:\M\to\M$ of $\V$-actegories.
It is then easy to check that $\under\oV\Eob{M}$ is left adjoint to $_E\M(\Eob{M},\under)$, so that $_E\M$ is a hommed $\V$-actegory.

The two strengths of the 2-variable $\V$-actegory morphism $\smpq$ are uniquely determined by those of $\smp$ due to the requirement that $q_{\Eob{M},\Eob{N}}$ be a transformation of actegory morphisms:
\begin{equation} \label{Gamma'^q}
\begin{CD}
V\oV(M\smp N)@>V\oV q_{M,\Eob{N}}>>V\oV(M\smpq\Eob{N})\\
@V{\Gamma'_{V,M,N}}VV @VV{\Gamma^{\prime q}_{V,M,\Eob{N}}=\phi\Gamma^{\prime q}_{V,\Eob{M},\Eob{N}}}V\\
M\smp(V\oV N)@>q_{M,V\oV\Eob{N}}>>M\smpq(V\oV\Eob{N})
\end{CD}
\end{equation}
and a similar diagram defines $\Gamma^q_{V,\Eob{M},\Eob{N}}:V\oV(\Eob{M}\smpq\Eob{N})\to(V\oV\Eob{M})\smpq\Eob{N}$.

The coherence conditions that $\gamma^q$, $\eta^q$ and $\eps^q$ should satisfy in order that they become transformations of actegory morphisms follow
from the analogous coherence conditions (\ref{sma-11}--\ref{sma-15}) for $\gamma$, $\eta$ and $\eps$ by using that $q_{\Eob{M},\Eob{N}}$ is a coequalizer,
$V\oV\under$ preserves such coequalizers and $\phi$ is faithful.

Since $\M$ is r2-exact, the monad $E\oV\under$ preserves reflexive coequalizers therefore its Eilenberg-Moore category $_E\M$ also has reflexive coequalizers.
Furthermore, $\Eob{M}\smpq\under$ preserves such coequalizers because both functors $M\smpq\,\phi\under$ and $E\oV(M\smpq\,\phi\under)$ occurring in (\ref{q coeq}) 
preserve them. The strength $\Gamma^{\prime q}$ in (\ref{Gamma'^q}) is an isomorphism since both horizontal arrows are coequalizers and $\Gamma'$ is an isomorphism.
Finally, $V\oV\under:\,_E\M\to\,_E\M$ also preserves reflexive coequalizers because $V\oV\under:\M\to\M$ does and $\phi$ creates reflexive coequalizers.
This proves r2-exactness of $\bra\,_E\M,\smpq,\Eob{R}\ket$.
\end{proof}

\textbf{Proof of Theorem \ref{thm: supplement of SMC}}
Applying the construction of Section \ref{sec: hot} there is an r2-exact skew monoidal tensored $\V$-category $(\,_E\M)^{T_q}$ of $T_q=\Eob{R}\smpq\under$-algebras on $_E\M$. Applying also the lifting procedure of Theorem \ref{thm: lift} to the skew monoidal functor $\phi:\,_E\M\to\M$ we obtain the $\bar\phi$ of Theorem \ref{thm: supplement of SMC}. By \cite[Theorem 5.3. (ii)]{SMC} this $\bar\phi$ is an equivalence
of categories therefore
it suffices to show strong skew monoidality of $\bar\phi$. The $\bar\phi_0$ is the identity and $\bar\phi_2$ is defined, as in Theorem \ref{thm: lift}, by the diagram
\begin{equation} \label{diag: phi2}
\parbox{300pt}{
\begin{picture}(300,120)
\put(0,95){$M\smp TN$}
\put(60,96){\vector(1,0){60}} \put(60,84){$M\smp(\phi\beta\ci q_{\Eob{N}})$}
\put(60,100){\vector(1,0){60}} \put(75,106){$\mu_{M,N}$}
\put(143,95){$ M\smp N$}
\put(190,98){\vector(1,0){60}} \put(205,105){$\pi_{M,\bar\phi\alg{N}^q}$}
\put(260,95){$M\hot\bar\phi\alg{N}^q$}

\put(10,80){\vector(0,-1){50}} \put(15,54){$q_{M,T^q\Eob{N}}\ci(M\smp q_{\Eob{N}})$}
\put(160,80){\vector(0,-1){50}} \put(165,54){$q_{M,\Eob{N}}$}
\dashline{3}(280,80)(280,32)\put(280,32){\vector(0,-1){0}} \put(285,54){$\bar\phi_{M,\alg{N}^q}$} 

\put(-5,10){$M\smpq T^q\Eob{N}$}
\put(65,11){\vector(1,0){60}} \put(75,0){$M\smpq \beta$}
\put(65,15){\vector(1,0){60}} \put(80,21){$\mu^q_{M,\Eob{N}}$}
\put(142,10){$ M\smpq\Eob{N}$}
\put(190,13){\vector(1,0){60}} \put(205,20){$\pi^q_{M,\alg{N}^q}$}
\put(265,10){$ M{\bar\smp}_q\alg{N}^q$}
\end{picture}}
\end{equation}
where $\alg{N}^q=\bra\Eob{N},\beta\ket$ is a $T^q$-algebra with underlying $E$-object $\Eob{N}$ and with action $\beta:R\smpq\Eob{N}\to\Eob{N}$, $M$ is an object of $\M$,
$q_{\Eob{N}}=q_{R,\Eob{N}}$ and ${\bar\smp}_q$ denotes the horizontal tensor product in $(\,_E\M)^{T_q}$. The $E$-object $\Eob{N}=\bra N,\lambda_{\Eob{N}}\ket$ is
related to the $T$-algebra $\alg{N}=\bar\phi\alg{N}^q$ by the monad morphism $\sigma_N$ because
\begin{align*}
\phi\beta\ci q_{\Eob{N}}\ci\sigma_N&=\phi\beta\ci q_{\Eob{N}}\ci\lambda_1[R\smp N]\ci(E\oV\eta_N)=\\
&=\phi\beta\ci\lambda_{R\smpq\Eob{N}}\ci(E\oV q_{\Eob{N}})\ci(E\oV\eta_N)=\\
&=\lambda_{\Eob{N}}\ci(E\oV \phi\beta)\ci(E\oV\eta^q_{\Eob{N}})=\lambda_{\Eob{N}}\,.
\end{align*}
Therefore composing the parallel pair of the 1st row in (\ref{diag: phi2}) with $M\smp\sigma_N$ we see that $\pi_{M,\alg{N}}\ci\rho=\pi_{M,\alg{N}}\ci\lambda_2$ therefore
$\pi_{M,\alg{N}}$ factorizes through $q_{M,\Eob{N}}$ by a unique $u: M\smpq\Eob{N}\to M\hot\alg{N}$. Then also 
$u$ factors uniquely through $\pi^q_{M,\alg{N}^q}$ and this defines an arrow which is necessarily the inverse of $\bar\phi_{M,\alg{N}^q}$.

\textbf{Acknowledgement}: The author wishes to thank the organizers, Edwin Beggs and Tomasz Brzezi\'nski, for the warm hospitality during the memorable  workshop "Categorical and Homological Methods in Hopf Algebras" held in Swansea, UK in 2013 December. 
This work was partially supported by the Hungarian Scientific Research Fund grant OTKA 108384. The author is indebted to the
referee for the crucial observations that essentially contributed to the content of Section \ref{monoidality}.

\end{document}